\documentclass[12pt,letterpaper,titlepage]{amsart}
\usepackage{amsmath, amssymb, amsthm, amsfonts,amscd,amsaddr, xr}
\usepackage[all]{xy}
\usepackage[usenames]{color}

\theoremstyle{plain}
\newtheorem{theorem}{Theorem}[section]
\newtheorem{proposition}[theorem]{Proposition}
\newtheorem{cor}[theorem]{Corollary}
\newtheorem{con}[theorem]{Conjecture}
\newtheorem{prop}[theorem]{Proposition}
\newtheorem{lemma}[theorem]{Lemma}
\newtheorem{definition}[theorem]{Definition}
\theoremstyle{definition}

\newtheorem{rmk}[theorem]{Remark}
\numberwithin{equation}{section}
\newtheorem{theoremA}{Theorem}

\newcommand{\bs}{\backslash}
\newcommand{\cc}{\mathcal{C}}
\newcommand{\C}{\mathbb{C}}

\newcommand{\A}{\mathcal{A}}
\newcommand{\Hc}{\mathcal{H}}
\newcommand{\Fc}{\mathcal{F}}

\newcommand{\Q}{\mathbb{Q}}
\newcommand{\Z}{\mathbb{Z}}

\newcommand{\R}{\mathbb{R}}
\newcommand{\N}{\mathbb{N}}

\newcommand{\SO}{\operatorname{SO}}

\newcommand{\GL}{\operatorname{GL}}

\newcommand{\Ad}{\operatorname{Ad}}
\newcommand{\ad}{\operatorname{ad}}
\newcommand{\diag}{\operatorname{diag}}
\newcommand{\vol}{\operatorname{vol}}

\newcommand{\supp}{\operatorname{supp}}

\newcommand{\err}{\operatorname{err}}
\newcommand{\rank}{\operatorname{rank}}

\newcommand{\PSl}{\operatorname{PSl}}
\newcommand{\re}{\operatorname{Re}}

\def\hat{\widehat}
\def\af{\mathfrak{a}}
\def\e{\epsilon}
\def\gf{\mathfrak{g}}
\def\cf{\mathfrak{c}}
\def\hf{\mathfrak{h}}
\def\kf{\mathfrak{k}}
\def\lf{\mathfrak{l}}

\def\pf{\mathfrak{p}}
\def\qf{\mathfrak{q}}

\def\sf{\mathfrak{s}}

\def\uf{\mathfrak{u}}

\def\la{\langle}
\def\ra{\rangle}
\def\1{{\bf1}}

\def\U{\mathcal{U}}
\def\B{\mathcal{B}}

\def\oline{\overline}
\def\F{\mathcal{F}}
\def\V{\mathcal{V}}

\def\propertyUI{{\rm (I)}}

\def\MTi{$\mathrm M\mathrm T$}
\def\WMTi{$\mathrm w\mathrm M\mathrm T$}

\title[Lattice counting]
{The harmonic analysis of lattice counting on real spherical spaces}

\begin{document}

\date{December 12, 2015}

\author[Kr\"otz]{Bernhard Kr\"{o}tz}
\email{bkroetz@gmx.de}
\address{Universit\"at Paderborn, Institut f\"ur Mathematik\\Warburger Stra\ss e 100, 
33098 Paderborn, Deutschland}
\thanks{The first author was supported by ERC Advanced Investigators Grant HARG 268105. The second author was partially supported by ISF 1138/10 and ERC 291612}
\author[Sayag]{Eitan Sayag}
\email{eitan.sayag@gmail.com}
\address{Department of Mathematics, Ben Gurion University of the Negev\\P.O.B. 653, Be'er Sheva 84105, Israel}
\author[Schlichtkrull]{Henrik Schlichtkrull}
\email{schlicht@math.ku.dk}
\address{University of Copenhagen, Department of Mathematics\\Universitetsparken 5, 
DK-2100 Copenhagen \O, Denmark}

\begin{abstract}By the collective name of {\it lattice counting} we refer to a setup introduced 
in \cite{DRS} that aims to establish a relationship between
arithmetic and randomness in the context of affine symmetric spaces.
In this paper we extend the geometric setup from symmetric 
to real spherical spaces and continue to develop the 
approach with harmonic analysis which was initiated in \cite{DRS}. 
\end{abstract}

\maketitle

\section{Introduction}

\subsection{Lattice counting}
Let us recall from Duke, Rudnick and Sarnak \cite{DRS} the setup of lattice counting on a
homogeneous space $Z=G/H$. Here $G$ is an algebraic 
real reductive group and $H<G$ an algebraic subgroup such that 
$Z$ carries an invariant measure. 
Further we are given a lattice $\Gamma<G$ such that its trace $\Gamma_H:=\Gamma\cap H$ 
in $H$ is a lattice in $H$. 

Attached to invariant measures $dh$ and $dg$ on $H$ and $G$ we 
obtain an invariant measure $d(gH)$ on $Z$ via Weil-integration:
$$ \int_Z \Big(\int_H f(gh) dh\Big)\  d(gH) = \int_G f(g) \ dg \qquad (f\in C_c(G))\, .$$  
Likewise the measures $dg$ and $dh$ give invariant measures $d(g\Gamma)$ and $d(h\Gamma_H)$ on 
$Y:=G/\Gamma$ and $Y_H:=H/\Gamma_H$. We pin down the measures $dg$ and $dh$ and hence $d(gH)$ by the request 
that $Y$ and $Y_H$ have volume one. 

\par Further we are given a family $\B$ of ``balls'' $B_R\subset Z$
depending on a parameter $R\ge 0$. 
At this point we are rather imprecise about the structure of these balls  and content 
us with the property that they constitute an exhausting family of compact sets as $R\to \infty$.   

\par Let $z_0=H\in Z$ be the standard base point. The {\it lattice
counting problem} for $\B$ consists of the determination of the
asymptotic behavior of the density of $\Gamma\cdot z_0$ in
balls $B_R\subset Z$, as the radius $R\to\infty$. By 
{\it main term counting}  for $\mathcal{B}$ we understand the statement that the asymptotic
density is 1. More precisely, 
with $$N_R(\Gamma, Z):= \#\{
\gamma\in \Gamma/\Gamma_H\mid \gamma\cdot z_0\in B_R\}$$
and $|B_R|:=\vol_Z (B_R)$ we say that main term counting holds if
\begin{equation}\label{mtc}
N_{R}(\Gamma, Z) \sim |B_{R}|\,\quad (R\to\infty).\end{equation}

\subsection{Relevant previous works}
The main term counting was established in \cite{DRS}
for symmetric spaces $G/H$ and 
certain families of balls,
for lattices with $Y_H$ compact. Furthermore, the main term counting in the case where $Y_{H}$ is non-compact 
was proven 
using a hypothesis on regularization of periods of Eisenstein series, whose proof remains unpublished.  
In subsequent work Eskin and McMullen
\cite{EM} removed the obstruction that $Y_H$ is compact and presented an ergodic
approach. Later Eskin, Mozes and Shah \cite{EMS} refined the
ergodic methods and discovered that main term counting holds for a
wider class of reductive spaces: For reductive algebraic groups $G, H$ 
defined over $\Q$ and arithmetic lattices $\Gamma<G(\Q)$  it is enough to request that 
the identity component of $H$ is not contained in a proper parabolic subgroup of $G$ 
which is defined over $\Q$ 
and that the balls $B_R$ satisfy a certain condition of {\it non-focusing}.

\par In these works the balls $B_R$ are constructed as follows. 
All spaces considered are affine in the sense that 
there exists a $G$-equivariant embedding of $Z$ into the representation module 
$V$ of a rational representation of $G$. For any such embedding and any norm on the vector 
space $V,$ one then obtains a family of balls $B_R$ on $Z$ by intersection with the metric balls 
in $V$. For symmetric spaces all families of balls produced this way
are suitable for the lattice counting, but in general one needs to 
assume non-focusing in addition.
In particular all maximal reductive subgroups 
satisfy all the conditions and hence fulfill the main term counting.

\subsection{Real spherical spaces}
In this paper we investigate the lattice counting 
for a real spherical space $Z$, that is, it is requested that 
the action of a minimal parabolic subgroups $P<G$ on $Z$ admits 
an open orbit. In addition we assume that $H$ is reductive and remark that 
with our standing assumption that $Z$ is unimodular
this is automatically satisfied 
for a spherical space when the Lie algebra $\hf$ of $H$ is self-normalizing (see \cite{KK}, Cor. 9.10). 

\par Our approach is based on spectral theory and is a natural 
continuation to \cite{DRS}. We consider a particular type of balls 
which are intrinsically defined by the geometry of $Z$ (and thus
not related to a particular representation $V$ as before).

\subsubsection{Factorization of spherical spaces}\label{Section: Factorization of spherical spaces}

\par In the spectral approach it is of relevance to get a control 
over intermediate subgroups $H<H^\star<G$ which arise in the following way: 
Given a unitary representation $(\pi, {\mathcal H})$ one looks at the smooth 
vectors ${\mathcal H}^\infty$ and its continuous dual  ${\mathcal H}^{-\infty}$, the 
distribution vectors. The space $({\mathcal H}^{-\infty})^H$
of $H$-invariant distribution vectors is
of fundamental importance. For all pairs $ (v,\eta)\in 
{\mathcal H}^\infty \times ({\mathcal H}^{-\infty} )^H$ one obtains a smooth function
on $Z$, a {\it generalized matrix-coefficient}, via
\begin{equation}\label{matrix coeff}
m_{v,\eta}(z) = \eta(g^{-1}\cdot v) \qquad (z=gH\in Z)\, .
\end{equation}
The functions (\ref{matrix coeff})
are the building blocks for the harmonic 
analysis on $Z$.  The stabilizer $H_\eta$ in $G$ of $\eta\in  
({\mathcal H}^{-\infty} )^H$ is a closed subgroup which contains 
$H$, but in general it can be larger than $H$ even if $\pi$ is non-trivial.

Let us call $Z^\star=G/H^\star$ a {\it factorization} of $Z$ if $H<H^\star$ and 
$Z^\star$ is unimodular. For a general real spherical space
$Z$ the homogeneous spaces $Z_\eta=G/H_\eta$ can happen to be non-unimodular
(see \cite{KKSS2} for $H$ the Iwasawa $N$-subgroup).  
However there is a large subclass of real spherical spaces which behave well 
under factorization. Let us call a factorization co-compact if $H^\star/H$ is 
compact and {\it basic} if (up to connected components) $H^\star$ is of the form $H_I:=HI$ for 
a normal subgroup $I\triangleleft G$.  
Finally we call a factorization {\it weakly basic} if it is obtained by a composition of 
a basic and a co-compact factorization. 

\subsubsection{Wavefront spherical spaces}
\par A real spherical space is called {\it wavefront} if the attached 
compression cone is a quotient of a closed Weyl-chamber.
The relevant definitions will be recalled in Section \ref{wfrss}.
Many real spherical spaces are wavefront: all 
symmetric spaces and all Gross-Prasad type spaces $G\times H/ H$  (see (\ref{GP1}) -  (\ref{GP3})) are
wavefront.\footnote{Also, if $Z$ is complex, then of the 78 cases in the list 
of \cite{BraPez}, the non-wavefront cases are (11), (24), (25), (27),  (39-50),  
(60), (61)}
The terminology {\it wavefront} originates from \cite{SV} because 
wavefront real spherical spaces satisfy the ``wavefront lemma'' of Eskin-McMullen 
(see \cite{EM}, \cite{KKSS1}) which is fundamental in the approach of 
\cite{EM} to lattice counting.  

\par On the geometric side wavefront real spherical spaces enjoy the following 
property from \cite{KKSS2}: All $Z_\eta$ are unimodular and 
the factorizations of the type 
$Z_\eta$ are precisely the weakly basic factorizations of $Z$. 

\par On the spectral level wavefront real spherical spaces are distinguished by the 
following integrability property, also from \cite{KKSS2}: The generalized matrix coefficients
$m_{v,\eta}$ of (\ref{matrix coeff})
belong to $L^p(Z_\eta)$ for some $1\leq p <\infty$ only depending on $\pi$ and $\eta$.

\subsubsection{Main term counting}
\par In the theorem below we assume that $Z$ is a wavefront real spherical space
of reductive type. 
For simplicity we also assume that all compact normal subgroups of $G$ are finite.
\par Using soft techniques from harmonic analysis and a general property of decay
from \cite{KSS1}, our first result (see Section \ref{Mt II}) is: 

\begin{theoremA}\label{thmA} 
Let $Z=G/H$ be  as above, and
assume that $Y=G/\Gamma$ is compact. Then main term counting (\ref{mtc}) holds. 
\end{theoremA}

Since wavefront real spherical spaces satisfy the wavefront lemma by \cite{KKSS1}, Section 6,
this theorem could also be derived with the ergodic method of \cite{EM}. 
In the current context the main point is thus the proof by harmonic analysis.

\par To remove the assumption that $Y$ is compact and to obtain error term bounds for the 
lattice counting problem we need to apply more sophisticated tools from harmonic analysis.
This will be discussed in the next paragraph
with some extra assumptions on $G/H$.

\subsection{Error Terms}

The problem of determining the error term in counting problems is notoriously difficult and
in many cases relies on deep arithmetic information. Sometimes, like
in the Gauss circle problem, some
error term is easy to establish but getting an optimal error term
is a very difficult problem.

\par We restrict ourselves to the cases where the cycle $H/\Gamma_H$
is compact.\footnote{After a theory for regularization of $H$-periods of Eisenstein series is developed, one can drop this assumption.} 
To simplify the exposition here we assume in addition that $\Gamma<G$ is
irreducible, i.e. there do not exist non-trivial  normal subgroups $G_1, G_2$  of $G$ and 
lattices $\Gamma_i < G_i$ such that $\Gamma_1 \Gamma_2$ has finite index in $\Gamma$. 

\par The error we study is measure theoretic in nature,
and will be denoted here as $\err(R,\Gamma)$.
Thus, $\err(R,\Gamma)$ measures the deviation of two measures on
$Y=\Gamma \backslash G$,
the counting measure arising from lattice points in a ball of radius $R$,
and the invariant measure $d \mu_{Y}$ on $Y$. More precisely, 
with $\1_R$ denoting the characteristic function of $B_R$ we
consider the densities $$F^{\Gamma}_{R}(g \Gamma):=
\frac{\sum_{\gamma\in \Gamma/ \Gamma_H} \1_R(g\gamma H)}{|B_{R}|}.$$   
Then, 
$$\err(R,\Gamma)=||F^{\Gamma}_{R}-d \mu_{Y}||_{1},$$ 
where  $||\,\cdot\,||_{1}$ denotes the total variation of the signed measure. Notice that 
$|F^{\Gamma}_{R}(e\Gamma)-1|=\frac{|N_{R}(\Gamma, Z) - |B_{R}||}{|B_{R}|}$ is essentially 
the error term for the pointwise count (\ref{mtc}).

Our results on the error term $\err(R,\Gamma)$ allows us to deduce results toward the error 
term in the smooth counting problem, a classical problem that studies the quantity 
 $$\err_{pt,\alpha}(R,\Gamma)=|B_R||F^{\Gamma}_{\alpha,R}(e\Gamma)-1|$$ where 
 $\alpha \in C_{c}^{\infty}(G)$ is a positive smooth function of compact support (with integral one) 
 and $F^{\Gamma}_{\alpha,R}=\alpha*F^{\Gamma}_{R}.$
See Remark \ref{error relate} for the comparison of $\err(R,\Gamma)$ with 
$\err_{pt,\alpha}(R,\Gamma)$.

\par To formulate our result we introduce the
exponent $p_H(\Gamma)$ (see (\ref{ic})), which measures the worst $L^{p}$-behavior
of any generalized matrix coefficient associated with a spherical
unitary representation $\pi$, which is  $H$-distinguished and occurs
in the automorphic spectrum of $L^2(\Gamma \backslash G)$. 
We first state our result for the non-symmetric case of triple product
spaces, which is Theorem \ref{thm=errorprasad} from the body
of the paper.

\begin{theoremA}\label{thmB} 
Let $Z=G_0^3/ \diag (G_0)$ for $G_0=\SO_e(1,n)$ 
and assume that $H/\Gamma_H$ is compact.
For all $p>p_H(\Gamma)$ there exists a $C=C(p)>0$ such that
$$\err(R,\Gamma)\leq C |B_R|^{-{1\over (6n+3)p}}\, $$
for all $R\geq 1$.
(In particular, main term counting holds in this case). 
Furthermore, in regards to smooth counting, for any $\alpha \in C_{c}^{\infty}(G)$ and 
for all $p>p_H(\Gamma)$ there exists a $C=C(p,\alpha)>0$ such that
$$\err_{pt,\alpha}(R,\Gamma)\leq C |B_R|^{1-{1\over (6n+3)p}}\, $$
for all $R\geq 1$.
\end{theoremA}

To the best of our knowledge this is the first error
term obtained for a non-symmetric space.  The crux of the
proof is locally uniform comparison between $L^{p}$
and $L^{\infty}$ norms of generalized matrix coefficients $m_{v,\eta}$
which is achieved by applying the model of
\cite{BRe} and \cite{Deit} for the triple product functional $\eta$ in spherical principal series.

\par It is possible to obtain error term bounds under a
certain technical hypothesis introduced in Section \ref{Section Lp-bounds}
and refered to as Hypothesis~A. This hypothesis in turn is implied by a conjecture on 
the analytic structure of families of Harish-Chandra modules which we explain in Section \ref{sec=conj}. 
The conjecture and hence the hypothesis appear to be 
true for symmetric spaces but requires quite a technical tour de force.
In general, the techniques  currently available do not allow for an elegant and efficient solution. 
Under this hypothesis we show that: 

\begin{theoremA}\label{thmC} 
Let $Z$ be wavefront real spherical space 
for which Hypothesis A  is valid.
Assume also
\begin{itemize}
\item $G$ is  semisimple with no compact factors
\item $\Gamma$ is arithmetic and irreducible
\item $\Gamma_{H}=H \cap \Gamma$ is co-compact in $H$.
\item $p>p_H(\Gamma)$
\item $k>{rank(G/K)+1\over 2} \dim(G/K) +1$
\end{itemize}
Then, there exists a constant $C=C(p, k)>0$ such that
$$\err(R,\Gamma)\leq C |B_R|^{-{1\over (2k+1)p}}\, $$
for all $R\geq 1$.
Moreover, if $Y=\Gamma \backslash G$ is compact one can replace the third condition by $k>\dim(G/K)+1$.
\end{theoremA}

The existence of a non-quantitative error
term for  symmetric spaces was established in
\cite{BeOh} and improved in \cite{GN}.

We note that in case of the hyperbolic plane our error term is still
far from the quality of the bound of A. Selberg. This is because
we only use a weak version of the trace formula, namely Weyl's law, and
use simple soft Sobolev bounds between eigenfunctions on $Y$.

\section{Reductive homogeneous spaces}\label{rhs}

In this section we review a few facts on reductive homogeneous spaces:
the Mostow decomposition, the associated geometric balls and their 
factorizations.

We use the convention that real Lie groups are denoted by upper case Latin letters, e.g $A, B, C$, 
and their Lie algebras by the corresponding lower case German letter $\af$, $\mathfrak{b}$, $\mathfrak{c}$. 
\par Throughout this paper $G$ will denote an algebraic real reductive group and $H<G$ 
is an algebraic subgroup.  We form the homogeneous space $Z=G/H$ and write $z_0=H$ for 
the standard base point. 

Furthermore, unless otherwise mentioned we assume that $H$ is reductive in $G$,
that is, the adjoint representation of $H$ on $\gf$ is completely reducible.
In this case we say that $G/H$ is {\it of reductive type}.

\par Let us fix a maximal compact subgroup $K<G$
for which we assume that the associated
Cartan involution $\theta$ leaves $H$ invariant (see the references to
\cite{KSS1}, Lemma 2.1).  
Attached to $\theta$ is the infinitesimal 
Cartan decomposition $\gf = \kf +\sf$ where $\sf = \kf ^\perp$ is the orthogonal complement 
with respect to a non-degenerate invariant bilinear form $\kappa$ on $\gf$ which is positive definite on $\sf$  
(if $\gf$ is semi-simple, then 
we can take for $\kappa$ the Cartan-Killing form).  Further we set $\qf:= \hf^\perp$.

\subsection{Mostow decomposition}

We recall Mostow's polar decomposition: 
\begin{equation} \label{Mostow} K \times _{H\cap K} \qf 
\cap \sf \to Z, \ \  [k,X]\mapsto k \exp(X) \cdot z_0 \end{equation}
which is a homeomorphism. 
With that we define  
$$\|k\exp(X)\cdot z_0\|_Z=\|X\|:= \kappa (X,X) ^{1\over 2}$$
for $k\in K$ and $X\in\qf\cap\sf$.

\subsection{Geometric balls}\label{balls}

The problem of lattice counting in $Z$
leads to a question of exhibiting natural
exhausting families of compact subsets. We use balls which 
are intrinsically defined by the geometry of $Z$.

We define the {\it intrinsic ball} of radius $R>0$ on $Z$ by
$$B_R:=\{z\in Z\mid \|z\|_Z<R\}\, .$$
Write $B_R^G$ for the intrinsic ball of $Z=G$, that is, if
$g=k\exp(X)$ with $k\in K$ and $X\in \sf$, then we put
$\|g\|_G=\|X\|$ and define $B^G_R$ accordingly.

\par Our first interest is the growth of the volume $|B_R|$ for $R\to \infty$. 
We have the following upper bound.

\begin{lemma}\label{volumes}
There exists a constant $c>0$ such that: 
$$|B_{R+r}|\leq e^{cr}|B_R|$$
for all $R\geq 1, r\geq 0$.
\end{lemma}

\begin{proof}  Recall the integral formula
\begin{equation}\label{int formula}
\int_Z f(z)\,dz=\int_{K} \int_{\qf\cap\sf}
f(k\exp(X).z_0)\delta(X)\,dX\,dk, 
\end{equation}
for $f\in C_c(Z)$,
where $\delta(Y)$ is the Jacobian at $(k,Y)$ of the
map (\ref{Mostow}). It is independent of $k$
because $dz$ is invariant. Then 
$$|B_R|=\int_{X\in\qf\cap\sf, \|X\|< R} \delta(X)\,dX\,.$$
Hence it suffices to prove that there exists $c>0$ such that
$$\int_0^{R+r} \delta(tX) t^{l-1} \,dt
\leq e^{cr} \int_0^R \delta(tX) t^{l-1}\, dt$$
for all $X\in\qf\cap\sf$ with $\|X\|=1$.
Here $l=\dim\qf\cap\sf$. Equivalently, the function
$$R\mapsto e^{-cR}\int_0^R  \delta(tX) t^{l-1}\, dt$$
is decreasing, or by differentiation,
$$\delta(RX) R^{l-1} \leq c \int_0^R \delta(tX) t^{l-1}\, dt$$
for all $R$. The latter inequality is established
in \cite[Lemma A.3]{EMS} with $c$ independent of $X$.
\end{proof}

Further we are interested how the volume behaves under distortion by elements 
from $G$.

\begin{lemma} \label{distortions}
%There exists a $\delta>0$ such that for all $\delta> \e>0$ and 
For all $r,R>0$ one has $ B_r^G B_R \subset B_{R+r}$. 
\end{lemma}

To prove the lemma we first record that:                    

\begin{lemma}\label{distance Z}
Let $z=gH\in Z$. Then $\|z\|_Z=\inf_{h\in H} \|gh\|_G$.
\end{lemma}

\begin{proof} It suffices to prove that
$\|\exp(X)h\|_G \geq \|X\|$ for $X\in\qf\cap\sf$, $h\in H$,
and by Cartan decomposition of $H$, we may assume
$h=\exp(T)$ with $T\in\hf\cap\sf$. Thus we have reduced
to the statement that
$$\|\exp(X)\exp(T)\|_G\geq \|\exp(X)\|_G$$
for $X\perp T$ in $\sf$. 
In order to see this, we note
that for each $g\in G$ the norm
$\|g\|_G$ is the length of the geodesic
in $K\bs G$ which joins
the origin $x_0$ to $x_0g$.
More generally the geodesic between
$x_0 g_1$ and $x_0 g_2$ has length $\|g_2g_1^{-1}\|_G$.
Hence $c=\|\exp(X)\exp(T)\|_G$
is the distance from
$A=x_0\exp(-T)$ to $B=x_0\exp(X)$.
As $X\perp T$ the points $A$ and $B$ form a right triangle with $C=x_0$.
The hypotenuse has length $c$ and the leg $CB$ has
length $a=\|\exp(X)\|$.
As the sectional curvatures are
non-positive we have $a^2+b^2\le c^2$. In particular $a\le c$.
\end{proof}

In particular, it follows that
\begin{equation}\label{distance inequality}
\|gz\|_Z\leq \|z\|_Z+\|g\|_G\quad (z\in Z, g\in G)
\end{equation}
and Lemma \ref{distortions} follows.

\begin{rmk}\label{compatible balls}
Observe that the norm $\|\,\cdot\,\|_G$ on $G$ depends on the chosen Cartan decomposition
$\theta$. However, by applying (\ref{distance inequality}) with $Z=G$ one sees that
the norm obtained with a conjugate $\theta'$ of $\theta$ will satisfy
\begin{equation}\label{two norms}
\|g\|'_G\le \|g\|_G+ c, \qquad \|g\|_G\le \|g\|'_G+ c'
\end{equation}
for all $g\in G$ with some constants $c,c'\ge 0$. 

For the definition of $\|\,\cdot\,\|_Z$ we assumed that $\theta$ leaves $H$
invariant. If instead we use the identity in Lemma \ref{distance Z} as the definition
of $\|\,\cdot\,\|_Z$ then this assumption can be avoided. In any case, it follows 
that the norms on $Z$ obtained from two different Cartan involutions will
satisfy similar inequalities as (\ref{two norms}). The 
corresponding families of balls are then also compatible,
$$B_R\subset B'_{R+c}, \qquad B'_R\subset B_{R+c'},$$
for all $R>0$.
\end{rmk}

\subsection{Factorization}\label{ball factorization}
 
By a (reductive) factorization of $Z=G/H$ we understand a homogeneous space 
$Z^\star = G/H^\star$ with $H^\star$ an algebraic subgroup of $G$ such that 
 
\begin{itemize} 
\item $ H^\star$ is reductive. 
\item $H\subset H^\star$. 
\end{itemize}

A factorization is called 
{\it compact} if $Z^\star$ is compact, and {\it co-compact}
if the fiber space ${\mathcal F}:=H^\star/H$ is compact.  
It is called {\it proper} if $\dim H<\dim H^\star<\dim G$.

\begin{lemma} \label{co-comp factor} Let $Z=G/H\to Z^\star=G/H^\star$ be a factorization. Then the following assertions are equivalent: 
\begin{enumerate}
\item $Z\to Z^\star$ is co-compact. 
\item There exist a compact subgroup $K^\star<H^\star$ such that $K^\star H= H^\star$.
\item There exists a compact subalgebra $\kf^\star<\hf^\star$ such that $\hf^\star= \kf^\star +\hf$ 
and $\exp(\kf^\star)<H^\star$ compact. 
\end{enumerate}
\end{lemma} 
\begin{proof} First (1) implies (2) by the Mostow decomposition of the 
reductive homogeneous space $H^\star/H$. Clearly (2) implies (3) as the multiplication map
$K^\star \times H \to H^\star$ needs to be submersive by Sard's theorem. 
Finally, for (3) implies (1)  we observe that $H^\star/H$ has finitely many components
and $\exp(\kf^\star)H$ is compact and open in there.
\end{proof}

\par Let  ${\mathcal F} \to  Z\to Z^\star$ be a factorization of $Z$. 
We write $B_R^\star$ and $\B_R^\Fc$ for the intrinsic balls in $Z^\star$ and 
$\Fc$, respectively. 

\begin{lemma}
We have
$B_R^\star= B_RH^\star/ H^\star$
and 
$B_R^\Fc
= B_R\cap \Fc.$
\end{lemma}

\begin{proof} Follows from Lemma \ref{distance Z}.
\end{proof}

\par For a compactly supported bounded measurable
function $\phi$ on $Z$ we define the
fiberwise integral
$$\phi^{\mathcal F}(gH^\star):= \int_{H^\star/H} \phi(gh^\star) \ d(h^\star H) $$
and recall the integration formula
\begin{equation}\label{integral with Z*}
\int_Z \phi(gH) \ d (gH)=\int_{Z^\star} \phi^{\mathcal F}(gH^\star) \ d(gH^\star)\,
\end{equation}
under appropriate normalization of measures.
Consider the characteristic function $\1_R$ of $B_R$
and note that its fiber average $\1_R^{\mathcal F}$ is supported in the compact ball
$B_R^\star$.
We say that the family of balls
$(B_R)_{R>0}$ {\it factorizes well to $Z^\star$} provided
for all compact subsets $Q\subset G$
\begin{equation} \label{factor limit}
\lim_{R\to \infty} {\sup_{g\in Q} \1_R^{\mathcal F}(gH^\star) \over |B_R|} =0\, .\end{equation}
Observe that for all compact subsets $Q$ there exists an 
$R_0=R_0(Q)>0$ such that
$$\sup_{g\in Q} \1_R^{\mathcal F}(gH^\star) \leq |B_{R+R_0}^{\mathcal F}|\, $$
by Lemma \ref{distortions}. 
Thus the balls $B_R$ factorize well provided
\begin{equation} \label{CRITERION} \lim_{R\to \infty}  {|B_{R+R_0}^{\mathcal F}|\over |B_R|} =0\, .\end{equation}
for all $R_0>0$. 

\begin{rmk}
The condition that the balls $B_{R}$ factorize well is closely related to the non-focusing condition (Definition 1.14 in \cite{EMS}). 
Thus, in the case of semi-simple connected $H$, the non-focusing condition of the intrinsic balls is implied 
by the condition that they factorize well to all factorizations. \end{rmk}

\subsection{Basic factorizations}\label{basicf}

There is a special class of factorizations with which we are dealing with in the sequel. 
From now on we assume that $\gf$ is semi-simple and write 
$$ \gf= \gf_1 \oplus \ldots \oplus \gf_m$$
for the decomposition  into simple ideals.  For a reductive subalgebra $\hf<\gf$ 
and a subset $I\subset \{ 1, \ldots, m\}$
we define the reductive subalgebra 
\begin{equation} \label{HI} \hf_I:= \hf + \gf_I=\hf + \bigoplus_{i\in I} \gf_i\, .\end{equation}
 
We say that the factorization is {\it basic} provided that $\hf^*=\hf_I$ for some $I$. 
Finally we call a factorization {\it weakly basic} if it is built from consecutive
basic and co-compact factorizations, that is, there exists a sequence
\begin{equation} \label{factor sequence} 
\hf^\star=\hf^k\supset\dots\supset\hf^0=\hf 
\end{equation}
 of reductive subalgebras
such that for each $i$ we have $\hf^i=(\hf^{i-1})_I$ for some $I$ or $\hf^i/\hf^{i-1}$ is compact.
The following lemma shows that in fact it suffices with $k\le 2$.

\begin{lemma}\label{two-step} 
Let $Z\to Z^\star$ be a weakly basic factorization.  
Then there exists an intermediate factorization $Z\to Z_b \to Z^\star$
such that $Z\to Z_b$ is basic and $Z_b\to Z^\star$ co-compact.
\end{lemma}
\begin{proof} 
Let a sequence (\ref{factor sequence}) of factorizations which are consecutively basic or compact be given. 
We first observe that two consecutive basic factorizations make up for 
a single basic factorization, and likewise two consecutive compact factorizations yield a 
single compact factorization by Lemma 
\ref{co-comp factor}. Hence it suffices to prove that
we can modify a string 
$$\hf^{i+2}\supset \hf^{i+1}\supset \hf^i$$ 
with $\hf^{i+2}/\hf^{i+1}$ basic and $\hf^{i+1}/\hf^i$ compact to 
$$\hf^{i+2}\supset\hf_b^{i+1} \supset \hf^{i}$$ with 
$\hf^{i+2}/\hf_b^{i+1}$ compact and $\hf_b^{i+1}/\hf^i$ basic.

We have $\hf^{i+2}=\hf^{i+1} +\gf_I$ for some $I$, and  
by Lemma \ref{co-comp factor}  that
$\hf^{i+1}= \hf^i +\cf$ with $\cf$ compact. 
Then $\hf_b^{i+1}:=\hf^i +\gf_I $ is a reductive subalgebra and a basic 
factorization of $\hf^i$. Furthermore $\hf^{i+2}=\hf_b^{i+1} +\cf$.  
This establishes the lemma
\end{proof}

\section{Wavefront real spherical spaces}\label{wfrss}

We assume that $Z$ is real spherical, i.e.~a minimal 
parabolic subgroup $P<G$ has an open orbit on $Z$. 
It is no loss of generality to assume that $PH\subset G$ is open, or equivalently
that $\gf=\hf +\pf$.  

\par If $L$ is a real algebraic group, then we write 
$L_{\rm n}$ for the normal subgroup of $L$ which is generated 
by all unipotent element. In case $L$ is reductive we observe that 
$\lf_{\rm n}$ is the sum of all non-compact simple ideals of $\lf$. 

\par According to \cite{KKS} there is a unique parabolic subgroup $Q\supset P$ 
with the 
following two properties: 
\begin{itemize}
\item $QH=PH$.
\item There is a Levi decomposition $Q=LU$ with $L_{\rm n} \subset Q\cap H\subset L$. 
\end{itemize}
Following \cite{KKS} we call $Q$ a $Z$-adapted parabolic subgroup.  

Having fixed $L$ we let $L=K_L A_L N_L$ be an Iwasawa decomposition of $L$.
We choose  an Iwasawa decomposition $G=KAN$ which inflates the one of $L$, i.e.
$K_L<K, A_L =A$ and $N_L<N$. 
Further we may assume that $N$ is the unipotent 
radical of the minimal parabolic $P$. 

\begin{rmk}\label{remark about K}
It should be noted that the assumption on the Cartan decomposition $\theta$,
which was demanded in Section \ref{balls}, may be overruled by 
the above requirement to $K$.
However, it follows from Remark \ref{compatible balls} that the balls $B_R$
can still be defined, and that the difference does not disturb the lattice counting on $Z$.
\end{rmk}
  
\par Set $A_H:=A\cap H$ and put $A_Z=A/A_H$. We recall that $\dim A_Z$ is an 
invariant of the real spherical space, called the real rank (see \cite{KKS}).   

In \cite{KKSS1}, Section 6, we defined the notion of {\it wavefront} for 
a real spherical space,
which we quickly recall. Attached to $Z$ is 
a geometric invariant, the so-called compression cone
which is a closed and convex subcone $\af_Z^-$ of $\af_Z$. 
It is defined as follows.
Write $\Sigma_\uf$ for the space of $\af$-weights of the $\af$-module $\uf$ and
let $\oline\uf$ denote the corresponding sum of root spaces for $-\Sigma_\uf$.
According to \cite{KKS} there exists a linear map
\begin{equation}\label{Tmap}
T: \oplus_{\alpha\in\Sigma_\uf} \gf^{-\alpha} = \oline \uf \,\to\,   
\lf_H^{\perp}\oplus \uf\subset \oplus_{\beta\in \{0\}\cup\Sigma_\uf} \gf^{\beta} 
\end{equation}
such that
$
\hf = \lf \cap \hf +\{ \oline X + T(\oline X)\mid \oline X \in \oline\uf\}.
$
Here  $\lf_H^{\perp}$ denotes the orthocomplement of $\lf \cap \hf$ in $\lf$.
For each pair $\alpha,\beta$ we denote by 
$$T_{\alpha,\beta}: \gf^{-\alpha}\to \gf^\beta$$
the map obtained from $T$ by restriction to
$\gf^{-\alpha}$ and projection to $\gf^\beta$.
Then $T=\sum_{\alpha,\beta} T_{\alpha,\beta}$ and by definition
$$\af_Z^-=\{ X\in \af\mid  (\alpha+\beta)(X)\ge 0, \,\,\forall \alpha,\beta
\text{ with } T_{\alpha,\beta} \neq 0\}.$$
It follows from (\ref{Tmap}) that $\alpha+\beta$ vanishes on $\af_H$
if $T_{\alpha,\beta}\neq 0$.
Hence $\af_Z^-\subset \af_Z$.
If one denotes by  $\af^{-}\subset\af $ the closure of the 
negative Weyl chamber, then 
$\af^-+\af_H\subset \af_Z^-$ and by definition
$Z$ is wavefront if
$$ \af^-+\af_H= \af_Z^- .$$

Let us mention that many real spherical spaces are wavefront; for example 
all symmetric spaces and all Gross-Prasad type spaces $Z=G\times H / H$ with $(G,H)$ one of the following
\begin{eqnarray}  
\label{GP1}& & (\GL_{n+1}(\C),  \GL_n(\C)),  \ (\GL_{n+1}(\R), \GL_n(\R)), \\
\label{GP2}& & (\GL_{n+1}(\mathbb{H}), \GL_n(\mathbb{H})), \  (\operatorname{U}(p+1,q), \operatorname{U}(p,q)), \\
\label{GP3}& & (\SO(n+1,\C), \SO(n,\C)), \ (\SO(p+1,q), \SO(p,q))\, .\end{eqnarray}

We recall from \cite{KKSS1} the polar decomposition for real spherical spaces 
\begin{equation}\label{polar} Z= \Omega A_Z^- F\cdot z_0\end{equation}
where 
\begin{itemize}
\item $\Omega$ is a compact set of the type $F'K $ with $F'\subset G$ a finite set. 
\item $F\subset G$ is a finite set with the property that $F\cdot z_0 = T \cdot z_0 \cap Z$
where $T=\exp(i\af)$ and the intersection is taken in $Z_\C= G_\C/H_\C$. 
\end{itemize}

\subsection{Volume growth} \label{Factorization}   
Define $\rho_Q\in\af^*$ by $\rho_Q(X) = {1\over 2} {\rm tr} (\ad_{\uf} X)$, 
$X\in \af$. It follows from the unimodularity of $Z$ and the local structure 
theorem that $\rho_Q|_{\af_H}=0$, i.e. 
$\rho_Q\in \af_Z^* =\af_H^\perp$.

\begin{lemma}\label{vg lemma} Let $Z=G/H$ be a wavefront real spherical space. 
Then 
\begin{equation}\label{asymp ball volume}
|B_R|\asymp \sup_{X\in \af\atop \|X\|\leq R} e^{2\rho_Q(X)}= \sup_{X\in \af_Z^-\atop 
\|X\|\leq R} e^{-2\rho_Q(X)}
\, .
\end{equation}
\end{lemma}

Here the expression $f(R)\asymp g(R)$ signifies that the ratio $\frac{f(R)}{g(R)}$ 
remains bounded below and above as $R$ tends to infinity.

\begin{proof} First note that the equality in (\ref{asymp ball volume}) is 
immediate from the wavefront assumption. 

\par Let us first show the lower bound, i.e. there exists 
a $C>0$ such that for all $R>0$ one has 
$$ |B_R|\geq C \sup_{X\in \af\atop \|X\|\leq R} e^{2\rho_Q(X)}\,.$$
For that we recall the volume bound from \cite{KKSS2}, Prop. 4.2: 
for all compact 
subsets $B\subset G$ with non-empty interior there exists a 
constant $C>0$ such that $\vol_Z (Ba\cdot z_0) \geq C a^{2\rho_Q}$
for all $a \in A_Z^-$. Together with the polar decomposition 
(\ref{polar}) this gives us the lower bound. 

\par As for the upper bound let 
$$\af_R^-:=\{ X\in \af^-\mid \|X\|\leq R\}\,  .$$ 
Observe that $B_R\subset B_R':= K A_R^- K \cdot z_0$. In the sequel it is convenient to realize 
$A_Z$ as a subgroup of $A$ (and not as quotient): we identify $A_Z$ with $A_H^\perp \subset A$. 
The upper bound will 
follow if we can show that 
$$|B_R'|\leq C \sup_{X\in \af\atop \|X\|\leq R} e^{2\rho_Q(X)}\qquad (R>0)\,.$$
for some constant $C>0$. This in turn will follow from the 
argument for the upper bound in the proof of Prop. 4.2 in \cite{KKSS2}:
in this proof we considered for $a\in A_Z^-$ the map 

$$\Phi_a: K \times \Omega_A \times \Xi \to G, \ \ (k, b, X )\mapsto kb
\exp(\Ad(a) X)$$
where $\Omega_A\subset A$ is a compact neighborhood of ${\bf 1}$
and $\Xi\subset \hf $
is a compact neighborhood of $0$. It was shown that the Jacobian of $\Phi_a$, 
that is $\sqrt{\det (d\Phi_a d\Phi_a^t)}$, is bounded by $Ca^{-2\rho_Q}$. 
Now this bounds holds as well for the right $K$-distorted 
map 
$$\Psi_a:  K \times \Omega_A \times K \times \Xi \to G, \ \ (k,b,k',X)
\mapsto kb
\exp(\Ad(ak') X)\, .$$

The reason for that comes from an inspection of the proof; all 
what is needed is the following fact: let $d:=\dim \hf$ and consider the 
action of $\Ad(a)$ on $V=\bigwedge^d \gf$. Then for $a\in A^-$ we have  
$$ a^{-2\rho} \geq \sup_{v\in V, \atop \|v\|=1} \la \Ad(a) v , v\ra\, .$$  
We deduce an upper bound 

\begin{equation}\label{rho-bound} \vol_Z(K \Omega_A a K \cdot z_0)\leq C a^{-2\rho}\, .\end{equation}
We need to improve that bound from $\rho$ to $\rho_Q$ on the right hand side of (\ref{rho-bound}).
For that let $W_L$ be the Weyl group of the reductive pair $(\lf, \af)$. 
Note that $\rho_Q={1\over |W_L|} \sum_{w\in W_L} w\cdot \rho$.  Further, the local structure 
theorem implies that $L_{\rm n}\subset H$ and hence $W_L$ can be realized as a 
subgroup of $W_{H\cap K}:= N_{H\cap K}(\af)/ Z_{H\cap K}(\af)$. We choose $\Omega_A$ to be invariant 
under $N_{H\cap K}(\af)$ and observe that $a\in A_Z$ is fixed under $W_{H\cap K}$.  Thus 
using the $N_{H\cap K}(\af)$-symmetry in the $a$-variable we refine (\ref{rho-bound}) to  
$$  \vol_Z(K \Omega_A a K \cdot z_0)\leq C a^{-2\rho_Q}\,.$$
The desired bound then follows. 
\end{proof}

\begin{cor} \label{factor} Let $Z=G/H$ be a wavefront real spherical space 
of reductive type.  Let $Z\to Z^\star$ be a basic factorization
such that $Z^\star$ is not compact. 
Then the geometric balls $B_R$ factorize well to $Z^\star$. 
\end{cor}

\begin{proof} 
As $Z\to Z^\star$ is basic we may assume (ignoring connected components)
that $H^\star= G_I H$ for some $I$. Note that 
${\mathcal F}=H^\star/H \simeq  G_I / G_I \cap H$ is real spherical.  

Let $Q$ be the $Z$-adapted parabolic subgroup attached to $P$.
Let $P_I = P \cap G_I$ and $ G_I \supset Q_I \supset P_I $ be the ${\mathcal F}$-adapted 
parabolic above $P_I$  and note that $Q_I=Q\cap G_I$. With Lemma \ref{vg lemma} we then get
$$|B_R^{\mathcal F}|\asymp \sup_{X\in \af_I\atop \|X\|\leq R} e^{2\rho_{Q_I}(X)}\, ,$$ 
which we are going to compare with (\ref{asymp ball volume}).

Let $\uf_I$ be the Lie algebra of the unipotent radical of $Q_I$. Note that $\uf_I \subset \uf$
and that this inclusion is strict since $G/H^\star$ is not compact.
The corollary now follows from (\ref{CRITERION}). 
\end{proof}

\subsection{Property I}

We briefly recall some results from \cite{KKSS2}.

\par Let $(\pi,\Hc_{\pi})$ be a unitary irreducible
representation of $G$.  We denote by $\Hc_\pi^{\infty}$ the $G$-Fr\'echet module of smooth vectors
and by $\Hc_\pi^{-\infty}$ its strong dual. One calls $\Hc_\pi^{-\infty}$ the $G$-module of 
distribution vectors; it is a DNF-space with continuous $G$-action. 

\par Let $\eta \in (\Hc_\pi ^{-\infty})^H$ be an $H$-fixed element and $H_\eta<G$ the 
stabilizer of $\eta$. 
Note that $H<H_\eta$ and set $Z_\eta:= G/H_\eta$.  
With regard to $\eta$ and $v\in \Hc^\infty$ we form the generalized matrix-coefficient 
$$m_{v,\eta}(gH):= \eta(\pi(g^{-1}) v) \qquad (g\in G)$$
which is a smooth function on $Z_\eta$.

We recall the following facts from \cite{KKSS2} Thm.~7.6 and Prop.~7.7:

\begin{prop} \label{propI}Let $Z$ be a wavefront real spherical space of reductive type. 
Then the following assertions hold: 
\begin{enumerate} 
\item\label{one} Every generalized matrix coefficient $m_{v,\eta}$ as above is bounded.
\item Let $H< H^\star<G$ be a closed subgroup such that $Z^\star$ is unimodular. Then 
$Z^\star$ is a weakly basic factorization. 
\item\label{two} Let $(\pi, \Hc)$ be a unitary irreducible representation of $G$ and 
let $\eta \in (\Hc_\pi^{-\infty})^H$.  Then: 
\begin{enumerate} 
\item $Z\to Z_\eta$ is a weakly basic factorization.  
\item $Z_\eta$ is unimodular and there exists $1\leq p <\infty$ such that
$m_{v,\eta}\in L^p(Z_\eta)$ for all $v \in \Hc_\pi^\infty$. 
\end{enumerate}
\end{enumerate}
\end{prop}

The property of $Z=G/H$ that (\ref{two}b) is valid for all 
$\pi$ and $\eta $ as above is denoted {\it Property (I)} in \cite{KKSS2}. 
Note that (\ref{one}) and (\ref{two}b) together imply $m_{v,\eta}\in L^{q}(Z_\eta)$ for $q>p$.
Assuming Property (I) we can then make the following notation.

\begin{definition}\label{defi p_H}
Given $\pi$ as above, define $p_{H}(\pi)$
as the smallest index $\ge 1$ such that
all $K$-finite generalized matrix coefficients $m_{v, \eta}$
with $\eta\in (\Hc_\pi^{-\infty})^H $
belong to $L^{p}(Z_\eta)$ for any $p>p_{H}(\pi)$.
\end{definition}

Notice that $m_{v,\eta}$ belongs to $L^p(Z_\eta)$ for
all $K$-finite vectors $v$ once that this is the case for
some non-trivial such vector $v$, see \cite{KKSS2} Lemma 7.2. 
For example, this could be the trivial $K$-type,
if it exists in $\pi$.

It follows from finite dimensionality of
$(\Hc_\pi^{-\infty})^H$ (see \cite{KS})
that $p_H(\pi)<\infty$.
We say that $\pi$ is $H$-tempered if $p_{H}(\pi)=2$.

\par The representation $\pi$ is said to be
$H$-distinguished if
$(\Hc_\pi^{-\infty})^H \ne \{0\}$.
Note that if $\pi$ is not $H$-distinguished
then $p_{H}(\pi)=1$.

\section{Lattice point counting: setup} \label{setup}

Let $G/H$ be a real algebraic homogeneous space. 
We further assume that we are given a lattice (a discrete subgroup with
finite covolume) $\Gamma\subset G$, such that $\Gamma_H:=\Gamma\cap H$ is
a lattice in $H$. We
normalize Haar measures on $G$ and $H$ such that:
\begin{itemize} \item $\vol (G/\Gamma)=1$.
\item $\vol (H/\Gamma_H)=1$.
\end{itemize}

\par Our concern is with the double fibration

$$ \xymatrix{ &G/\Gamma_H \ar[ld] \ar[rd] & \\ 
Z:=G/H & & Y:=G/\Gamma& }$$
Fibre-wise integration yields transfer maps from functions on
$Z$ to functions
on $Y$ and vice versa. In more precision,
\begin{equation} \label{I1}
L^\infty (Y)\to L^\infty(Z), \ \phi\mapsto \phi^H; \ \phi^H(gH)
:=\int_{H/ \Gamma_H} \phi(gh\Gamma)\ d(h\Gamma_H)\end{equation}
and we record that this map is contractive, i.e
\begin{equation} \label{I1a} \|\phi^H\|_\infty \leq \|\phi\|_\infty
\qquad (\phi \in L^\infty(Y))\, .\end{equation}
Likewise we have
\begin{equation}\label{I2}  L^1(Z) \to L^1(Y), \ f\mapsto f^\Gamma;
\ f^\Gamma(g\Gamma):=\sum_{\gamma\in
\Gamma/ \Gamma_H} f(g\gamma H)\, , \end{equation}
which is contractive, i.e
\begin{equation} \label{I2a} \|f^\Gamma\|_1 \leq \|f\|_1
\qquad (f \in L^1(Z))\, .\end{equation}
Unfolding with respect to the double fibration yields, in view of our
normalization of measures,  the following adjointness relation:
\begin{equation} \label{ufo} \la f^\Gamma, \phi \ra_{L^2(Y)}=
 \la f, \phi^H\ra_{L^2(Z)}
\end{equation}
for all $\phi\in L^\infty (Y)$ and $f\in L^1(Z)$.
Let us note that (\ref{ufo}) applied
to $|f|$ and $\phi=\1_{Y}$ readily yields
(\ref{I2a}).

We write $\1_R \in L^1(Z)$ for the characteristic function of
$B_R$ and deduce from the definitions and (\ref{ufo}):
\begin{itemize}
\item $\1_R^\Gamma(e\Gamma)=N_R(\Gamma, Z):=
\#\{ \gamma\in \Gamma/\Gamma_H\mid \gamma\cdot z_0\in B_R\}$.
\item $\|\1_R^\Gamma\|_{L^1(G/\Gamma)}= |B_R|$.
\end{itemize}

\subsection{Weak asymptotics}

In the above setup, $G/H$ need not be of reductive type, but we shall
assume this again from now on. 
For spaces with property {\propertyUI}  and $Y$ compact
we prove analytically  in the following section that 
\begin{equation}\tag{\MTi}\label{counting}
N_{R}(\Gamma, Z) \sim |B_{R}|\,\quad
(R\to\infty)\, .\end{equation}
For that we will use the following result of \cite{KSS1}:

\begin{theorem}\label{th=1} Let $Z=G/H$ be of reductive type. 
The smooth vectors for the regular representation
of $G$ on $L^p(Z)$ vanish at infinity, for all $1\le p<\infty $.
\end{theorem}

\par With notation from (\ref{I2}) we set
$$F_R^\Gamma:= {1\over |B_{R}|} \1_R^\Gamma.$$
We shall concentrate on verifying the following limit of weak type:
\begin{equation}\tag{\WMTi}\label{eq1a}  \la F_R^\Gamma, \phi\ra_{L^2(Y)} \to
\int_Y \bar\phi\, d \mu_{Y}
\quad (R\to\infty),\qquad (\forall\phi \in C_0(Y))\,.
\end{equation}
Here $C_0$ indicates functions vanishing at infinity.

\begin{lemma} {\rm (\ref{eq1a})} $\Rightarrow$
{\rm (\ref{counting})}.
\end{lemma}

\begin{proof} As in \cite{DRS} Lemma 2.3 this is deduced from Lemma \ref{volumes}
and Lemma \ref{distortions}. 
\end{proof}

\section{Main term counting}\label{Mt II}
In this section we will establish main term counting
under the mandate of property \propertyUI\ and
$Y$ being compact. Let us call a family of balls 
$(B_R)_{R>0}$ 
{\it well factorizable} if it factorizes well to all 
proper factorizations of type $Z\to Z_\eta$.

\subsection{Main theorem on counting}

\begin{theorem}\label{leading term theorem}
Let $G$ be semi-simple and $H$ a closed reductive 
subgroup.
Suppose that $Y$ is compact and $Z$ admits \propertyUI.
If $(B_R)_{R>0}$ is well factorizable, then 
{\rm (\ref{eq1a})} and  {\rm (\ref{counting})} hold.
\end{theorem}

\begin{rmk}  In case $Z=G/H$ is real spherical and wavefront, 
then $Z$ has \propertyUI\ by Proposition  \ref{propI}.  If we assume in addition
that $G$ has no compact factors and that 
all proper factorizations are basic, then the family of 
geometric balls is well factorizable by Corollary \ref{factor}. 
In particular, Theorem \ref{thmA} of 
the introduction then follows from the above.
\end{rmk}

\par The proof is based on the following proposition.
For a function space $\F(Y)$ consisting of integrable functions on
$Y$ we denote by $\F(Y)_{\rm van}$ the subspace of functions with vanishing
integral over $Y$.

\begin{prop}\label{prop: criterion}
Let $Z=G/H$ be of reductive type. Assume that
there exists a dense subspace
$\mathcal{A}(Y) \subset
C_{b}(Y)^{K}_{\rm van}$ such that
\begin{equation}\label{PH}
\phi^{H}\in C_0(Z) \quad \hbox{for all $\phi\in
\mathcal{A}(Y)$}\,.
\end{equation}
Then {\rm (\ref{eq1a})} holds true.
\end{prop}

\begin{proof} We will establish (\ref{eq1a})
for $\phi\in C_b(Y)$.
As $$C_b(Y)=C_b(Y)_{\rm van}\oplus \C\1_Y,$$ and
(\ref{eq1a}) is trivial for $\phi$ a constant,
it suffices to establish
\begin{equation}\label{eq1b}
\la F_R^\Gamma, \phi\ra_{L^2(Y)} \to 0 \qquad (\phi \in C_b(Y)_{\rm van})\, .
\end{equation}
We will show (\ref{eq1b}) is valid for $\phi \in \mathcal{A}(Y)$.
By density, as $F_R^\Gamma$ is $K$-invariant and belongs to $L^{1}(Y)$,
this will finish the proof.

Let $\phi\in\mathcal{A}(Y)$ and let $\epsilon>0$.
By the unfolding identity (\ref{ufo}) we have
\begin{equation}\label{equnfold}
 \la F_R^\Gamma, \phi \ra_{L^2(Y)}= \frac{1}{|B_R|} \la
\1_R, \phi^H\ra_{L^2(Z)}.
\end{equation}
Using (\ref{PH}) we choose
$K_{\epsilon} \subset Z$ compact such
that $|\phi^{H}(z)|<\epsilon$ outside of $K_{\epsilon}$. Then
\begin{equation*} \frac{1}{|B_R|} \la \1_R, \phi^{H}\ra_{L^2(Z)}=
\int_{K_{\epsilon}}
+ \int_{Z-K_{\epsilon}}\quad
\frac{\1_R(z)}{|B_R|}\phi^{H}(z)\,d\mu_{Z}(z)\, .\end{equation*}
By (\ref{I1a}), the first term is bounded by
$\frac{|K_{\epsilon}|||\phi||_{\infty}}{|B_R|}$, which
is $\leq \epsilon$ for $R$ sufficiently large.
As the second term is bounded by $\epsilon$ for all $R$,
we obtain (\ref{eq1b}). Hence (\ref{eq1a}) holds.
\end{proof}

\begin{rmk}\label{first remark about UI'}
It is possible to replace (\ref{PH}) by a weaker requirement:
Suppose that an algebraic sum
\begin{equation}\label{deco A(Y)}
 \A(Y)= \sum_{j\in J} \A(Y)_j
\end{equation}
is given
together with a factorization
$Z^\star_j=G/H^\star_j$ for each $j\in J$. Suppose that the balls $B_R$  all 
factorize well to $Z^\star_j$, $j\in J$. 
Suppose further that $\phi^H$ 
factorizes to a function 
\begin{equation}\label{PH*} \phi^{H_j^\star}\in C_0(Z_j^\star)\, \end{equation}
for all $\phi\in \A(Y)_j$ and all $j\in J$.
Then the conclusion
in Proposition \ref{prop: criterion} is still valid. 
In fact, using (\ref{integral with Z*})
the last part of the proof modifies to:
\begin{align*} \frac{1}{|B_R|} \la \1_R, \phi^{H}\ra_{L^2(Z)} &=
\frac{1}{|B_R|} \la \1_R^\mathcal{F}, \phi^{H^\star_j}\ra_{L^2(Z^\star_j)}=\\
&= \int_{K_{\epsilon}^\star}
+ \int_{Z^\star_j-K_{\epsilon}^\star}\quad
\frac{\1_R^\mathcal{F}(z)}{|B_R|}\phi^{H^\star_j}(z)\,d\mu_{Z^\star_j}(z)\, \end{align*}
for $\phi\in \A(Y)_j$.
As $\|1_R^\mathcal{F}\|_{L^1(Z^\star_j)}=|B_R|$,  the second term is bounded by $\epsilon$
for all $R$. As the balls factorize well 
to $Z^\star_j$ we get 
the  first term as small as we wish with
(\ref{factor limit}).
\end{rmk}

\subsection{The space $\A(Y)$}

We now construct a specific subspace $\A(Y)\subset C_b(Y)_{\rm van}^K$ 
and verify condition (\ref{PH*}).

\par Denote by $\hat G_s \subset \hat G$ the $K$-spherical unitary dual. 
\par As $Y$ is compact, the abstract Plancherel-theorem implies:
$$L^2(G/\Gamma)^K \simeq \bigoplus_{\pi \in \hat G_s} (\Hc_\pi^{-\infty})^\Gamma.$$
If we denote the Fourier transform
by
$f\mapsto f^\wedge$ then the corresponding
inversion formula is given by
\begin{equation}\label{inversion}
f=\sum_{\pi} a_{v_\pi,f^\wedge(\pi)}.
\end{equation}
Here $a_{v_\pi,f^\wedge(\pi)}$ denotes a matrix coefficient for $Y$
with $v_\pi\in \Hc_\pi$ normalized $K$-fixed
and $f^\wedge(\pi)\in (\Hc_\pi^{-\infty})^\Gamma$, 
and the sum in (\ref{inversion})
is required to include multiplicities.
The
matrix coefficients for $Y$ are defined as in (\ref{matrix coeff}), 
that is
\begin{equation}
a_{v,\nu}(y) = \nu(g^{-1}\cdot v) \qquad (y=gH\in Y)\, .
\end{equation}
for $v\in \Hc_\pi$ and $\nu\in (\Hc_\pi^{-\infty})^\Gamma$.

Note that $L^2(Y)= L^2(Y)_{\rm van} \oplus \C \cdot \1_Y $. We define
$\mathcal{A}(Y)\subset L^2(Y)_{\rm van}^K$ to be the dense subspace
of functions with finite Fourier support, that is, 
$$\A(Y)=\operatorname{span}\{ a_{v,\nu}  \mid
\pi\in \hat G_s \text{ non-trivial}, v\in \Hc_\pi^K, \nu\in (\Hc_\pi^{-\infty})^\Gamma\}.$$
Then $\A(Y)\subset L^2(Y)^{K,\infty}_{\rm van}$ is dense
and since $C^\infty(Y)$ and $L^2(Y)^\infty$ are topologically
isomorphic, it follows that
$\mathcal{A}(Y)$ is dense in $C(Y)^K_{\rm van}$
as required.

The following lemma together with 
Remark \ref{first remark about UI'} immediately
implies Theorem \ref{leading term theorem}.

\begin{lemma}\label{existA} Assume that $Y$ is compact
and $Z$ has \propertyUI, and define $\A(Y)$ as above.
Then there exists a decomposition of $\A(Y)$ satisfying
{\rm (\ref{deco A(Y)})-(\ref{PH*})}. 
\end{lemma}

\begin{proof} 
The map
$\phi\mapsto \phi^H$ from (\ref{I1})
corresponds on the spectral side to a map
$(\Hc_\pi^{-\infty})^\Gamma \to (\Hc_\pi^{-\infty})^H$, which can
be constructed as follows.

As  $H/ \Gamma_H$ is
compact, we can define for each $\pi\in\hat G_s$
\begin{equation} \label{vint}
\Lambda_{\pi}: (\Hc_\pi^{-\infty})^\Gamma \to (\Hc_\pi^{-\infty})^H, \ \
\Lambda_{\pi}(\nu)=\int_{H/\Gamma_H} \nu\circ \pi(h^{-1}) \ d
(h\Gamma_H)\end{equation}
by $\Hc_\pi^{-\infty}$-valued integration: the defining integral is understood
as integration over a compact fundamental domain $F\subset H$ with respect
to the Haar measure on $H$; as the integrand is continuous and
$\Hc_\pi^{-\infty}$ is a complete locally convex space,
the integral converges in $\Hc_\pi^{-\infty}$. It follows from (\ref{vint})
that $(a_{v,\nu})^H=m_{v,\Lambda_\pi(\nu)}$ for all $v\in\Hc_\pi^\infty$ and
$\nu\in (\Hc_\pi^{-\infty})^\Gamma$.

Let $J$ denote the set of all factorizations $Z^\star\to Z$, 
including also $Z^\star=Z$ which we give the index $j_0\in J$.
For $j\in J$ we
define $\A(Y)_j\subset \A(Y)$ accordingly to
be spanned by the matrix coefficients $a_{v,\nu}$
for which $H_{\Lambda_\pi(\nu)}=H^\star_j$. Then (\ref{deco A(Y)}) holds.

\par Let $\phi\in\mathcal{A}(Y)_{j_0}$, then it follows
from (\ref{inversion}) that
\begin{equation}\label{Fourier phi^H}
\phi^H =\sum_{\pi\neq \1} m_{v_\pi,\Lambda_\pi(\phi^\wedge(\pi))} \,.
\end{equation}
Note that $H_\eta=H$ for each distribution vector $\eta={\Lambda_\pi(\phi^\wedge(\pi))}$ in this sum,
by the definition of $\mathcal{A}(Y)_{j_0}$.
As $Z$ has property \propertyUI{} the summand $m_{v_\pi,\Lambda_\pi(\phi^\wedge(\pi))}$
is contained in
$L^p(G/H)$ for $p>p_H(\pi)$, and by \cite{KKSS2}, Lemma 7.2, this containment is then
valid for all $K$-finite generalized matrix
coefficients $m_{v,\Lambda_\pi(\phi^\wedge(\pi))}$ of $\pi$.
Thus $m_{v_\pi,\Lambda_\pi(\phi^\wedge(\pi))}$
generates a Harish-Chandra module inside $L^p(G/H)$. As $m_{v_\pi,\Lambda_\pi(\phi^\wedge(\pi))}$
is $K$-finite, we conclude that it is a smooth vector. Hence
$\phi^H\in L^p(G/H)^\infty$, and in view of Theorem \ref{th=1} we obtain (\ref{PH}).
\par 
The proof of (\ref{PH*})
for $\phi\in \A(Y)_j$ for general $j\in J$ is obtained by the same reasoning,
where one replaces  $H$ by $H_j^\star$ in (\ref{vint}) and (\ref{Fourier phi^H}).
\end{proof}

This concludes the proof of
Theorem \ref{leading term theorem}.

\section{$L^p$-bounds for generalized matrix coefficients}\label{Section Lp-bounds}

From here on we assume that $Z=G/H$ is wavefront and real spherical. 
Recall that we assumed that $G$ is semi-simple and that we wrote 
$\gf=\gf_1\oplus\ldots\oplus\gf_m$ for the decomposition 
of $\gf$ into simple factors. It is no big loss of generality 
to assume that $G=G_1\times \ldots \times G_m$ splits accordingly. 
We will assume that from now on. 

\par Further we request that the lattice $\Gamma<G$ is irreducible, that 
is,  the projection of $\Gamma$ to any normal subgroup $J\subsetneq G$
is dense in $J$.

\par Let $\pi$ be an irreducible unitary representation of $G$. Then 
$\pi=\pi_1 \otimes\ldots\otimes \pi_m$ with $\pi_j$ and irreducible 
unitary representation of $G_j$.  We start with a simple 
observation. 

\begin{lemma} \label{pi-trivial}Let $(\pi, {\mathcal H})$ be an irreducible unitary 
representation of $G$ and $0\neq \nu \in ({\mathcal H}^{-\infty})^\Gamma$. 
If one constituent $\pi_j$ of $\pi$ is trivial, then $\pi$ is 
trivial. 
\end{lemma}

\begin{proof} 
The element $\nu$ gives rise to   
a $G$-equivariant injection  
\begin{equation}\label{nu-coeff} 
{\mathcal H}^\infty \hookrightarrow  C^\infty(Y), 
\ \ v\mapsto (g\Gamma\mapsto \nu (\pi(g^{-1})v))\, .\end{equation}
Say $\pi_j$ is trivial and let $J:= \prod_{i=1\atop i\neq j}^m G_i$.
Let $\Gamma_J$ be the projection of $\Gamma$ to $J$. Then 
(\ref{nu-coeff}) gives rise to a $J$-equivariant injection 
${\mathcal H}^\infty \hookrightarrow  C^\infty(J/\Gamma_J)$. 
As $\Gamma_J$ is dense in $J$, the assertion follows. 
\end{proof}

We assume from now on that the cycle $H/\Gamma_H\subset Y$ is compact.
This technical condition ensures that the vector valued average 
map (\ref{vint}) 
converges. 

\begin{lemma}\label{finite quotient} 
Let $(\pi, {\mathcal H})$ be a non-trivial irreducible unitary 
representation of $G$. Let $\nu \in (\Hc_\pi^{-\infty})^\Gamma$ such that 
$\eta:=\Lambda_\pi(\nu)\in (\Hc_\pi^{-\infty})^H$ is non-zero. 
Then $H_\eta/H$ is compact. 
\end{lemma}

\begin{proof}
Recall from Proposition \ref{propI} that 
$Z\to Z_\eta$ is weakly basic,
and from Lemma \ref{two-step} that then there exists $H\subset H_b \subset H_\eta$ such that
$H_\eta/H_b$ is compact and $Z\to Z_b$ is basic. Hence
$\hf_b=\hf_I$ for some $I$.
As $\pi$ is irreducible it infinitesimally embeds into
$C^\infty(Z_\eta)$ and hence also to $C^\infty(Z_b)$ on which $G_i$ acts trivially
for $i\in I$.
It follows that $\pi_i$ is trivial for $i\in I$. Hence Lemma \ref{pi-trivial} implies
$I=\emptyset$ and thus $\hf_b=\hf$.
\end{proof}

In the sequel we use the Plancherel theorem
(see \cite{HC})
$$L^2(G/\Gamma)^K \simeq \int_{\hat G_s}^\oplus  \V_{\pi,\Gamma}
\ d \mu(\pi)\,, $$
where $\V_{\pi,\Gamma}\subset (\Hc_\pi^{-\infty})^\Gamma$ is a finite
dimensional subspace and of constant dimension on each
connected component in the continuous spectrum (parametrization by
Eisenstein series), and where the Plancherel measure $\mu$
has support
$$\hat G_{\Gamma, s}:=\supp(\mu)\subset\hat G_s\, .$$

\par Given an irreducible lattice $\Gamma \subset G$ we
define
(cf.~Definition \ref{defi p_H})
\begin{equation}\label{ic} 
p_H(\Gamma):=\sup\{p_H(\pi): \pi\in\hat G_{\Gamma,s} \}
\end{equation}
and record the following.

\begin{lemma}\label{pHG} 
Assume that $G=G_1 \times\ldots\times G_m$ with all $\gf_i$ simple and non-compact. Then
$p_H(\Gamma)< \infty$.
\end{lemma}

\begin{proof}  For a unitary representation $(\pi,\Hc)$ and vectors $v,w\in \Hc$ we form the 
matrix coefficient $\pi_{v,w}(g):= \la \pi(g)v, w\ra$. We first claim that there exists 
a $p<\infty$ (in general depending on $\Gamma$) such that for  all non-trivial $\pi \in \hat G_{\Gamma,s}$  one 
has $\pi_{v,w}\in L^p(G)$ for all $K$-finite vectors $v,w$. In case $G$ has property (T) 
this follows (independently of $\Gamma$) from \cite{Cowling}. 
The remaining cases contain at least one factor $G_i$ of 
$\SO_e(n,1)$ or $\operatorname{SU}(n,1)$
(up to covering) and have no compact factors by assumption. They are treated in \cite{C}. 

\par The claim can be interpreted geometrically via the leading exponent $\Lambda_V\in \af^*$  which is attached 
to the Harish-Chandra module of $\Hc$ (see \cite{KKSS2}, Section 6). The lemma now follows from 
Prop.~4.2 and Thm.~6.3 in \cite{KKSS2}  (see the proof of Thm.~7.6 in \cite{KKSS2} how these two facts combine to 
result in integrability).
\end{proof}

Let $1\leq p<\infty$. Let us say that a subset $\Lambda \subset \hat G_s $ 
is {\it $L^p$-bounded} provided
that $m_{v, \eta} \in L^p(Z_\eta)$ for all $\pi\in \Lambda$ and
$v\in \Hc_\pi^\infty$, $\eta\in (\Hc_\pi^{-\infty})^H$.
By definition we thus have that $\hat G_{\Gamma,s}$ is $L^p$-bounded
for $p>p_H(\Gamma)$.

In this section we work under the following:

\bigskip 
\par\noindent{\bf Hypothesis A:} {\it  
For every $1\leq p<\infty$ and every $L^p$-bounded subset $\Lambda\subset \hat G_s$
there exists a compact subset $\Omega\subset G$ and constants $c, C>0$ 
such that the following assertions hold
for all $\pi\in \Lambda$, $\eta\in (\Hc_\pi^{-\infty})^H$ and
$v\in \Hc_\pi^K$:
\begin{equation}\tag{A1} \label{hypo1}  \|m_{v, \eta}\|_{L^p(Z_\eta)}  \leq C  
\|m_{v, \eta}\|_\infty\, ,\end{equation}
\begin{equation}\tag{A2} \label{hypo2} \|m_{v, \eta}\|_\infty  \leq c \| m_{v, \eta}\|_{\infty, \Omega_\eta}
\,  \end{equation}
where $\Omega_\eta= \Omega H_\eta/ H_\eta$. 
Here $ \| \,\cdot\, \|_{\infty, \omega}$ denotes the supremum norm taken on the subset $\omega.$ }

\bigskip  In the sequel we are only interested in the following 
choice of subset $\Lambda\subset \hat G_s$, namely
\begin{equation}\label{defi Lambda}
\Lambda:=\{\pi \in \hat G_{\Gamma,s}\mid \Lambda_\pi(\nu)\neq 0\  \hbox{for some}
\ \nu \in {\mathcal V}_{\pi, \Gamma}\}\,.
\end{equation}

An immediate consequence of Hypothesis A is: 

\begin{lemma}\label{fiber} Assume that $p>p_H(\Gamma)$. Then there is a $C>0$ such that 
for all 
$\pi\in \hat G_{\Gamma,s} $,  $v\in\Hc_\pi^K$,
$\nu\in (\Hc_\pi^{-\infty})^\Gamma$ and $\eta:=\Lambda_\pi(\nu)\in (\Hc_\pi^{-\infty})^H$
one has 
$$\|\phi_{\pi}^{H}\|_{L^p(Z_\eta)} \leq C\|\phi_{\pi}\|_{\infty}$$
where  $\phi_\pi(g\Gamma):= \nu(\pi(g^{-1})v)$. 
\end{lemma}

\begin{proof} Recall from (\ref{I1a}), that
integration is a bounded operator from $L^{\infty}(Y)
\to L^{\infty}(Z)$. Hence the assertion follows from (\ref{hypo1}).
\end{proof}

Recall the Cartan-Killing form $\kappa$ on $\gf=\kf+\sf$ and choose a basis $X_1, \ldots, X_l$ of $\kf$ and
$X_1', \ldots, X_s'$ of $\sf$ such that $\kappa(X_i, X_j)=-\delta_{ij}$ and $\kappa(X_i', X_j')=\delta_{ij}$.
With that data we form the standard Casimir element
$$\cc:=-\sum_{j=1}^l\ X_j^2 + \sum_{j=1}^s (X_j')^2 \in \U(\gf)\, .$$
Set $\Delta_K:=\sum_{j=1}^l\ X_j^2 \in \U(\kf)$ and obtain
the commonly used Laplace
element
\begin{equation} \label{Laplace} \Delta=\cc+2\Delta_K \in \U(\gf)\,
\end{equation}
which acts on $Y=G/\Gamma$ from the left.

Let $d \in \N$. For $1\leq p\leq \infty$, it follows from
\cite{BK}, Section 3,  that Sobolev norms on $L^p(Y)^\infty\subset
C^{\infty}(Y)$ can be defined by
$$||f||_{p,2d}^{2}=\sum_{j=0}^{d}||\Delta^{j}f||_p^{2}\, .$$
Basic spectral theory allows one to define $\|\cdot\|_{p, d}$ 
more generally for any $d\geq 0$.

\par Let us define
$$s:=\dim \sf = \dim G/K =\dim \Gamma\bs G/K $$
and
$$r:=\dim \af =\rank_\R(G/K)\, ,$$
where $\af\subset\sf$ is maximal abelian.

\par We denote by $C_b(Y)$ the space of continuous bounded functions on $Y$ and by 
$C_b(Y)_{\rm van}$ the subspace with vanishing integral.

\begin{proposition}\label{mpro} Assume that
\begin{enumerate}
\item $Z$ is a wavefront real spherical 
space, 
\item $G=G_1 \times\ldots\times G_m$ with all $\gf_i$ simple and non-compact.
\item $\Gamma<G$ is irreducible and 
$Y_H$ is compact, 
\item Hypothesis A is valid. 
\end{enumerate}
Let $p>p_H(\Gamma)$. Then the  map
$${\rm Av_{H}}:
C^{\infty}_b(Y)^K_{\rm van} \to L^{p}(Z)^K; {\rm Av_{H}}(\phi)=\phi^{H}$$
is continuous. More precisely, for all
\begin{enumerate}
\item $k> s+1$ if $Y$ is compact.
\item $k>{r+1\over 2} s +1$ if $Y$ is non-compact and $\Gamma$ is arithmetic
\end{enumerate}
there exists a constant
$C=C(p, k)>0$ such that
$$\|\phi^{H}\|_{L^p(Z)} \leq
C \|\phi\|_{\infty,k} \qquad (\phi \in
C^{\infty}_b(Y)_{\rm van}^K)\, $$
\end{proposition}

\begin{proof}  For all $\pi \in \hat G$
the operator $d\pi(\cc)$ acts as a scalar $\lambda_\pi$ and we set 
$$|\pi|:= |\lambda_\pi|\geq 0\, .$$
Let $\phi\in C_b^\infty(Y)^K_{\rm van}$ and write
$\phi=\phi_d +\phi_c$ for its decomposition in discrete and continuous
Plancherel parts.
We assume first that $\phi=\phi_d$.

\par In case $Y$ is compact we have Weyl's law:
There is a constant $c_Y>0$ such that
$$ \sum_{|\pi|\leq R} m(\pi)
 \sim c_Y R^{s/2} \qquad (R\to \infty)\, .$$
Here $m(\pi)= \dim \V_{\pi, \Gamma}$.
We conclude that
\begin{equation}\label{weight}
\sum_{\pi} m(\pi) (1+|\pi|)^{-k} < \infty
\end{equation}
for all $k> s/2+ 1$.
In case $Y$ is non-compact, we let $\hat G_{\mu,d}$ be the
the discrete support of the Plancherel measure. Then
assuming $\Gamma$ is arithmetic, the upper bound in \cite{Ji}
reads:
$$ \sum_{\pi \in \hat G_{\mu,d} \atop |\pi|\leq R} m(\pi)
 \leq  c_Y R^{rs/2} \qquad (R>0)\, .$$
For $k> rs/2+ 1$ we obtain (\ref{weight}) as before.

\par Let $p>p_H(\Gamma)$.
As $\phi$ is in the discrete spectrum we  decompose it
as $\phi=\sum_\pi \phi_{\pi}$
and obtain by Lemmas \ref{finite quotient} and \ref{fiber}
$$\|\phi^{H}\|_{p} \leq \sum_\pi \|\phi_{\pi}^{H}\|_{p}
\leq C \sum_\pi \|\phi_{\pi}\|_{\infty}\, .$$
The last sum we estimate as follows:
\begin{eqnarray*}
\sum_\pi \|\phi_{\pi}\|_{\infty}&=&\sum_\pi
(1+|\pi|)^{-k/2}(1+|\pi|)^{k/2} \|\phi_{\pi}\|_{\infty}\\
&\leq& C \sum_\pi  (1+|\pi|)^{-k/2}\|\phi_{\pi}\|_{\infty,k}\end{eqnarray*}
with $C>0$ a constant depending only on $k$ (we allow universal positive constants to change
from line to line).
Applying the Cauchy-Schwartz inequality combined with (\ref{weight})
we obtain
$$\|\phi^{H}\|_{p} \leq C \Big(\sum_{\pi} \|\phi_{\pi}\|_{\infty,k}^{2}\Big)^{\frac{1}{2}}\, $$
with $C>0$. With Hypothesis (\ref{hypo2}) 
we get the further improvement: 
$$\|\phi^{H}\|_{p} \leq C \Big(\sum_{\pi} \|\phi_{\pi}\|_{\Omega, \infty,k}^{2}\Big)^{\frac{1}{2}}\, $$
where the Sobolev norm is taken only over the compact set  $\Omega$.

To finish the proof we apply the Sobolev lemma on $K\backslash G$.
Here Sobolev norms are defined by the central operator $\cc$,
whose action agrees with the left action of $\Delta$.
It follows that $\|f\|_{\infty, \Omega} \leq C \|f\|_{2,k_1, \Omega}$ with
$k_1>\frac{s}{2}$ for $K$-invariant functions $f$ on $G$.
This gives $$\|\phi^{H}\|_{p}
\leq C(\sum_{\pi}||\phi_{\pi}||_{\Omega, 2,k+k_1}^{2})^{\frac{1}{2}} =
C||\phi||_{\Omega, 2,k+k_1} \leq C ||\phi||_{\infty,k+k_1}$$
which proves the proposition for the discrete spectrum.

\par If $\phi=\phi_c$ belongs to the continuous spectrum,
where multiplicities are
bounded (see \cite{HC}), the proof is simpler.
Let $\mu_c$ be the restriction of the Plancherel measure
to the continuous spectrum. As this is
just Euclidean measure on $r$-dimensional space we have
\begin{equation}\label{mu_c}
\int_{\hat G_s} (1+|\pi|)^{-k} \ d\mu_c(\pi)<\infty
\end{equation}
if $k>r/2$.
We assume for simplicity in what follows that $m(\pi)=1$
for almost all $\pi\in\supp\mu_c$.
As $\sup_{\pi \in \supp{\mu}_c} m(\pi) <\infty$
the proof is easily adapted to the general case.

Let $$\phi=\int_{\hat G_s} \phi_\pi \ d\mu_c(\pi).$$
As $\|\phi^H\|_\infty \leq \|\phi\|_\infty$ we conclude
with Lemma \ref{fiber}, (\ref{mu_c}) and Fubini's theorem
that
$$\phi^H =\int_{\hat G_s} \phi_\pi^H \ d\mu_c(\pi)\, $$
and, by the similar chain of inequalities as in the
discrete case
$$
\|\phi^H\|_p \leq C \|\phi\|_{\infty,k+k_1}
$$
with $k>\frac r2$ and
$k_1>\frac s2$.
This concludes the proof.
\end{proof}

\section{Error term estimates}

Recall $\1_R$,  the characteristic function of $B_R$.   
The first error term for the lattice counting problem 
can be expressed by
$$\err(R, \Gamma):= \sup_{\phi\in C_b(Y)\atop \|\phi\|_\infty\leq 1}
|\left\la {\1_R^\Gamma\over |B_R|} - \1_Y, \phi \right\ra| \qquad (R>0),$$
and our goal is to give an upper bound for $\err(R, \Gamma)$
as a function of $R$.

According to the decomposition $C_b(Y)= C_b(Y)_{\rm van} \oplus \C \1_Y$
we decompose functions as $\phi=\phi_o +\phi_1$ and obtain
\begin{equation*}
\err(R, \Gamma) =
\sup_{\phi\in C_b(Y)\atop \|\phi\|_\infty\leq 1} {|\la \1_R^\Gamma, \phi_o\ra|\over|B_R|}
= \sup_{\phi\in C_b(Y)\atop \|\phi\|_\infty\leq 1} {|\la \1_R,
\phi_o^H\ra|\over |B_R|}\, .
\end{equation*}
Further, from
$\|\phi_o\|_\infty\leq 2 \|\phi\|_\infty$ we obtain that
$\err(R,\Gamma)\leq 2 \err_1(R, \Gamma)$ with
$$\err_1(R, \Gamma):=
\sup_{\phi\in C_b(Y)_{\rm van}\atop \|\phi\|_\infty\leq 1} {|\la \1_R^\Gamma, \phi\ra|\over|B_R|}
=  \sup_{\phi\in C_b(Y)_{\rm van}\atop \|\phi\|_\infty\leq 1} {|\la \1_R, \phi^H\ra|\over |B_R|}\, .$$

\subsection{Smooth versus non-smooth counting}

Like in the classical Gauss circle problem one obtains much better
estimates for the remainder term if one uses a smooth cutoff.
Let $\alpha\in C_c^\infty(G)$ be a non-negative test function with normalized integral.
Set $\1_{R, \alpha}:= \alpha * \1_R$ and define
$$\err_{\alpha}(R, \Gamma):=  \sup_{\phi\in C_b(Y)^{K}_o\atop \|\phi\|_\infty\leq 1}
{|\la \1_{R,\alpha}^\Gamma, \phi\ra| \over |B_R|}=  \sup_{\phi\in C_b(Y)^{K}_o\atop \|\phi\|_\infty\leq 1}
{|\la \1_{R,\alpha}, \phi^H\ra| \over |B_R|}\, .$$

\begin{lemma}\label{smooth counting lemma}
Let $k>s +1$ if $Y$ is compact and
$k>{r+1\over 2} s +1$ otherwise. Let $p> p_H(\Gamma)$ and $q$ be such that ${1\over p} + {1\over q} =1$.
Then there exists $C>0$ such that
\begin{equation}\label{alpha error bound}
\err_{\alpha}(R,\Gamma) \leq  C \|\alpha\|_{1, k} |B_R|^{-{1\over p}}
\end{equation}
for all $R\geq 1$ and all $\alpha\in C_c^\infty(G)$.
\end{lemma}

\begin{proof} First note that
$$\la \1_{R, \alpha}, \phi^H\ra =\la \1_{R, \alpha}, (-\1+\Delta)^{k/2}(-\1+\Delta)^{-k/2}\phi^H\ra\, .$$
With $\psi=(-\1+\Delta)^{-k/2}\phi$ we have
$\|\psi\|_{\infty,k} \leq C\|\phi\|_\infty$
for some $C>0$. We thus
obtain
\begin{eqnarray*} \err_{\alpha}(R, \Gamma)& \leq C& \sup_{\psi\in C_b(Y)^{K}_o\atop \|\psi\|_{\infty,k} \leq 1}
{|\la \1_{R,\alpha}, (-\1+\Delta)^{k/2}\psi^{H} \ra| \over |B_R|}\\
& \leq & {C\over |B_R|} \sup_{\psi\in C_b(Y)^{K}_o\atop \|\psi\|_{\infty,k} \leq 1}
|\la \1_{R,\alpha}, (-\1+\Delta)^{k/2}\psi^{H} \ra| \end{eqnarray*}

Moving $(-\1+\Delta)^{k/2}$ to the other side we  get with H\"older's inequality and
Proposition \ref{mpro} that
\begin{equation*} \err_{\alpha}(R, \Gamma) \leq
{C\over |B_R|}   \|(-\1+\Delta)^{k/2}\alpha * \1_{R}||_{q}\, . 
\end{equation*}

Finally,
$$\|(-\1+\Delta)^{k/2}\alpha * \1_{R}\|_{q}\leq C \|\alpha\|_{1, k} \|\1_R\|_q$$
and with $\|\1_R\|_q= |B_R|^{1\over q}$, 
the lemma follows.
\end{proof}

\begin{rmk} \label{error relate}In the literature results are 
sometimes stated not with respect to $\err(R, \Gamma)$ but 
the pointwise error term $\err_{pt} (R,\Gamma) = |\1_R^\Gamma (\1) - |B_R||$. 
Likewise we define $\err_{pt,\alpha} (R,\Gamma)$.
Let $B_Y$ be a compact neighborhood of $\1 \Gamma\in Y$ and note 
that 
$$\err_{pt,\alpha}(R, \Gamma)\leq |B_R| \sup_{\phi\in L^1(B_Y)\atop \|\phi\|_1\leq 1}
|\la {\1_{R,\alpha}^\Gamma\over |B_R|} - \1_Y, \phi \ra| \qquad (R>0).$$
The Sobolev estimate  $\|\phi\|_\infty\leq C \|\phi\|_{1, k}$, for 
$K$-invariant functions $\phi$ on $B_Y$ and with
$k=\dim Y/K$ the Sobolev shift, 
then relates these error terms:
$$\err_{pt,\alpha}(R, \Gamma)\leq  |B_R| \sup_{\phi\in C^\infty_b(Y)\atop \|\phi\|_{\infty, -k}\leq 1}
|\la {\1_R^\Gamma\over |B_R|} - \1_Y, \phi \ra| \, .$$
We then obtain
$$\err_{pt,\alpha} (R,\Gamma)\leq C |B_R|^{1-{1\over p}} \qquad (R>0) \, $$
in view of (\ref{alpha error bound}).
\end{rmk}

We return to the error bound in Lemma \ref{smooth counting lemma} and would like to
compare $\err_1(R, \Gamma)$ with $\err_\alpha(R, \Gamma)$. For that 
we note (by the triangle inequality) that
%let $ C_b(Y)^{K}_{o,\R}$ denote the subspace of real valued functions in  $C_b(Y)^{K}_o$. Then 
%$$\err_{\alpha}(R, \Gamma):=  \sup_{\phi\in C_b(Y)^{K}_{o,\R}\atop \|\phi\|_\infty\leq 1}
%{\la \1_{R,\alpha}, \phi^H\ra \over |B_R|}$$
%and likewise for $\err_1(R, \Gamma)$ . Hence 
$$|\err_1(R, \Gamma) -\err_\alpha(R, \Gamma)|\leq
\sup_{\phi\in C_b(Y)^{K}_o\atop \|\phi\|_\infty\leq 1} { |\la \1_{R, \alpha}^\Gamma-\1_R^\Gamma, \phi\ra|\over |B_R|}\, .$$
Suppose that $\supp \alpha \subset B_\e^G$ for some $\e>0$. 
Then Lemma \ref{distortions} implies that $\1_{R,\alpha}$ is
supported in $B_{R+\epsilon}$, and hence
\begin{eqnarray*} |\la \1_{R, \alpha}^\Gamma-\1_R^\Gamma, \phi\ra| &\leq& \|\1_{R, \alpha}^\Gamma-\1_R^\Gamma\|_1\\
&\leq& \|\1_{R, \alpha}-\1_R\|_1\\
&\leq& |B_{R+ \e}|^{1\over 2}  \|\1_{R, \alpha}-\1_R\|_2 \\
&\leq& |B_{R+ \e}|^{1\over 2}  |B_{R+ \e} \bs B_R|^{1\over 2} \, . \end{eqnarray*}
With Lemma \ref{volumes}   we get 
$$|B_{R+\e} \bs B_R|\leq  C\e  |B_R|   \qquad (R\geq 1, \e<1)\, .$$
Thus we obtain that
$$|\err_1(R, \Gamma) -\err_\alpha(R, \Gamma)|\leq C \e^{\frac12} \, .$$
Combining this with the estimate in Lemma
\ref{smooth counting lemma} we arrive at
the existence of $C>0$ such that
$$\err_1(R, \Gamma)\leq  C (\e^{-k} |B_R|^{-{1\over p}}+  \e^{\frac12}) \, $$
for all $R\geq 1$ and all $0<\e<1$.
The minimum of the function $\e\mapsto \e^{-k} c +\e^{1/2}$ is
attained at $\e= (2kc)^{2\over 2k+1}$
and thus we get:

\begin{theorem} \label{thm=error} Under the assumptions of Proposition \ref{mpro}
the first error term $\err(R,\Gamma)$ for the lattice counting
problem on $Z=G/H$ can be estimated as follows: for all $p>p_H(\Gamma)$ and
$k>s +1$ for $Y$ compact, resp.  $k>{r+1\over 2} s +1$ otherwise,
there exists a constant $C=C(p, k)>0$ such that
$$\err(R,\Gamma)\leq C |B_R|^{-{1\over (2k+1)p}}\, $$
for all $R\geq 1$.
\end{theorem}

\begin{rmk} The point where we lose essential information is in the estimate
(\ref{weight})
where we used Weyl's law. In the moment pointwise multiplicity
bounds are available
the estimate would improve.
To compare the results with Selberg on the hyperbolic disc,
let us assume that
$p_H(\Gamma)=2$. Then with $r=1$ and $s=2$
our bound is $\err(R,\Gamma)\leq C_\e  |B_R|^{-{1\over 14}+\e}$
while Selberg showed
$\err(R,\Gamma)\leq C_\e  |B_R|^{-{1\over 3}+\e}$.
\end{rmk}

\section{Triple spaces}\label{UI for triple}

In this section we verify our Hypothesis A for triple space $Z= G/H$ where 
$G=G'\times G' \times G'$,  $H = \diag(G')$ and 
$G'=\SO_e(1, n)$ for some $n\geq 2$. Observe that $\SO_e(1,2) \cong \PSl(2,\R)$. 
We take $K':= \SO(n, \R)<G'$ as a maximal compact subgroup and set 
$K:= K' \times K ' \times K'$. Further we set $\sf:= \sf' \times \sf' \times \sf'$.
A maximal abelian subspace $\af \subset \sf$ is then of the form 

$$ \af = \af_1' \times \af_2' \times \af_3'$$
with $\af_i'\subset \sf'$ one dimensional subspaces. 
We recall the following result from \cite{DKS}.

\begin{prop} For the triple space the following assertion hold true: 
\begin{enumerate} 
\item $G=KAH$ if and only if $\dim (\af_1' +\af_2' +\af_3')=2$. 
\item Suppose that all $\af_i'$ are pairwise distinct. Then one has $PH$ is open 
for all minimal parabolics $P$ with Langlands-decomposition 
$P=M_P A_P N_P$ and $A_P=A$. 
\end{enumerate}
\end{prop}

We say that the choice of $A$ is {\it generic} if all $\af_i'$ are distinct and 
$\dim (\af_1' +\af_2' +\af_3')=2$. 

The invariant measure $dz$ on $Z$ can then be estimated as 
$$ \int_Z f(z) \ dz \leq 
\int_K \int_A f(ka\cdot z_0) J(a) \ \ da \ dk \qquad (f\in C_c(Z), f\geq 0)$$
with 
\begin{equation} \label{Jacobi}  J(a)  = \sup_{w\in W} a^{2w\rho} \end{equation}  
by Lemma \ref{vg lemma}. 
Note that in this case the Weyl group $W$ is just $\{\pm1\}^3$.

\subsection{Proof of the Hypothesis A}

We first note that for all $\pi \in \hat G_s$ 
the space of $H$-invariants  
$$ (\Hc_\pi^{-\infty})^H = \C I \, .$$
is one-dimensional, see \cite{CO}, Thm. 3.1.

\par Write $\pi=\pi_1
\otimes \pi_2\otimes \pi_3$ with each factor a $K'$-spherical
unitary irreducible representation of $G'$. If we assume that $\pi\neq \1$ has
non-trivial  $H$-fixed distribution vectors, then at least two of
the factors $\pi_i$ are non-trivial.

\par  Let $v_i$ be normalized $K'$-fixed vectors of
$\pi_i$ and set $v=v_1\otimes v_2\otimes v_3$. Since $Z$ is a
multiplicity one space, the functional $I\in(\Hc_\pi^{-\infty})^H$
is unique up to scalars. Our concern is to obtain uniform 
$L^p$-bounds for the generalized matrix coefficients
$f_\pi:=m_{v,I}$:
$$f_\pi(g_1, g_2, g_3):=
I(\pi_1(g_1)^{-1} v_1 \otimes \pi_2(g_2)^{-1}v_2 \otimes
\pi_3(g_3)^{-1}v_3)\, ,$$ when $\pi$ belongs to 
the set $\Lambda$ of (\ref{defi Lambda}).

\par We decompose $\Lambda= \Lambda_0 \cup \Lambda_1\cup \{{\bf \1}\}$
with $\Lambda_0\subset \Lambda$ the set of $\pi\in \Lambda$ with all $\pi_i$ 
non-trivial, and $\Lambda_1$ the set of $\pi$'s with exactly one $\pi_i$ to 
be trivial.

\par Consider first the case where $\pi \in \Lambda_1$, i.e.
one $\pi_i$ is
trivial, say $\pi_3$. Then $\pi_2=\pi_1^*$.
We identify $Z \simeq G'\times G'$ via $(g, h) \mapsto (\1,g,h)H$ and obtain
$$f_\pi(g, h)= \la \pi_1(g)v_1, v_1\ra\, , $$
a spherical function. Note that $Z_\eta\simeq G'$ 
and Hypothesis A follows from standard 
properties  about $K'$-spherical functions on $G'$. To be more specific 
let $G'=N'A'K'$ be an Iwasawa-decomposition with middle-projection 
${\bf a}: G'\to A'$, then 
$$f_\pi(g, h)= \varphi_{\lambda_1}(g):= 
\int_{K'} {\bf a}(k'g)^{\lambda_1-\rho'} \ dk'\, .$$
We use Harish-Chandra's estimates 
$|\varphi_\nu(a)|\leq a^\nu \varphi_0(a)$ and
$\varphi_0(a)\le C a^{-\rho}(1+|\log a|)^d$ for $a\in A'$ in positive chamber. 
The condition of $\pi\in \Lambda_1$ implies that 
$\rho - \re \lambda_1>0$ is bounded away from zero and Hypothesis A follows in this case.

\par Suppose now that $\pi \in \Lambda_0$, i.e. all $\pi_i$ are non-trivial.
\par For a simplified exposition we assume that $n=2$, i.e. $G'=\PSl(2,\R)$, 
and comment at the end for the general case.  
Then $\pi_i =\pi_{\lambda_i}$ are principal series 
for some $\lambda_i \in i\R^+ \cup [0, 1)$ with $\Hc_{\pi_{i}}^\infty= C^\infty({\mathbb S}^1)$
in the compact realization.  Set $\lambda=(\lambda_1, \lambda_2, \lambda_3)$ and 
set $\pi=\pi_\lambda$.

\par In order to analyze $f_\pi$ we use $G=KAH$ and thus assume that
$g=a=(a_1, a_2, a_3)\in A$. We work in the compact model of
$\Hc_{\pi_i}= L^2({\mathbb S}^1)$ and use the explicit model for
$I$ in \cite{BRe}: for $h_1, h_2, h_3$ smooth functions on the
circle one has

\begin{align*}I(h_1 \otimes h_2 \otimes h_3)=& {1\over (2 \pi)^3}
\int_0^{2\pi} \int_0^{2\pi} \int_0^{2\pi} h_1(\theta_1) h_2(\theta_2)
h_3(\theta_3)\cdot \\
&\cdot  {\mathcal K}(\theta_1, \theta_2, \theta_3) \ d\theta_1 d\theta_2
d\theta_3 \, ,\end{align*}
where

$${\mathcal K}(\theta_1, \theta_2, \theta_3)=|\sin(\theta_2
-\theta_3)|^{(\alpha-1)/2}|\sin(\theta_1
-\theta_3)|^{(\beta-1)/2}|\sin(\theta_1
-\theta_2)|^{(\gamma-1)/2}\,.$$ In this formula one has
$\alpha=\lambda_1-\lambda_2 -\lambda_3$, $\beta=-\lambda_1
+\lambda_2 -\lambda_3$ and $\gamma=-\lambda_1-\lambda_2
+\lambda_3$ where $\lambda_i \in i\R \cup(-1, 1)$ are the
standard representation parameters of $\pi_i$. 
According to to \cite{CO}, Cor. 2.1, the kernel ${\mathcal K}$ is absolutely 
integrable.

\par Set 
$$A':=\Big\{ a_t:=\begin{pmatrix}  t & 0 \\ 0 & {1\over t}\end{pmatrix} \mid t>0\Big \} < G'$$ 
Then $A_i' = k_{\phi_i} A' k_{\phi_i}^{-1}$ with $\phi_i \in [0, 2\pi]$ and 

$$k_{\phi} = \begin{pmatrix} \cos \phi & -\sin\phi \\ \sin \phi & \cos\phi\end{pmatrix}\, .$$ 
Set $a_{t, i}=  k_{\phi_i} a_t k_{\phi_i}^{-1}$.

Returning to our analysis of $f_\pi$ we now take
$h_i(t_i, \theta_i) =[\pi_1(a_{t_i, i}) v_i](\theta_i)$
and remark that 
$$h_i(t_i, \theta_i)= {1\over ( t_i^2 + \sin^2(\theta_i -\phi_i)( {1\over t_i^2} -
  t_i^2))^{{1\over 2}(1+\lambda_i)}}\, .$$

\par Let us set $|\pi|:= \pi_{\re \lambda_1} \otimes \pi_{\re\lambda_2} \otimes \pi_{\re \lambda_3}$.
Our formulas then show 
\begin{equation} \label{bound} |f_\pi(a)|\leq 
f_{|\pi|}(a)\qquad (a\in A) \, .\end{equation}

\par Let $c_i:= 1-|\re \lambda_i|$ for $i=1,2,3$.  
The fundamental estimate in \cite{KSS2}, Thm.~3.2, 
then yields a constant $d$, 
independent of $\pi$, and a constant $C=C(\pi)>0$
  such that for $a=(a_{t_1,1}, a_{t_2,2}, a_{t_3, 3})$ one has 
\begin{equation} \label{triple bound}
|f_\pi(a)|\leq C { (1 +|\log t_1| + |\log t_2| + |\log t_3|)^d  \over [\cosh \log t_1]^{c_1}\cdot
[\cosh \log t_2]^{c_2} \cdot [\cosh\log t_3]^{c_3}}\, .\end{equation}
In view of (\ref{bound}) the constant $C(\pi)$ 
can be assumed to depend only on  the distance of $\re \lambda_i$ to the
trivial representation. Looking at  the integral representation of $f_\pi$ with the kernel ${\mathcal K}$ we deduce a lower bound 
without the logarithmic factor, i.e. the bound is essentially sharp. 
Hence (\ref{Jacobi}) together with the fact that all $f_\pi$ for $\pi \in \Lambda_0$ are in  $L^p(Z)$ for some $p<\infty$ 
implies that 
\begin{equation} \label{inf-bound} \inf_{\pi \in \Lambda_0} c_i(\pi)>0\, . \end{equation}

We now claim 
\begin{equation} \label{p-bound} \sup_{\pi \in \Lambda_0} \|f_\pi\|_p<\infty\, , \end{equation}
and
\begin{equation} \label{infty-bound} \sup_{\pi \in \Lambda_0} \|f_\pi\|_\infty<\infty\, .\end{equation}

For $0<\e<1$ set $\Lambda_{\e,\R}=[0,1-\e] \times [0,1-\e]\times[0,1-\e]$ and $\Lambda_\e:=i\af^* +\Lambda_\e$. 
It follows from (\ref{inf-bound}) that there exists an $\e>0$ such that $\Lambda_0\subset \Lambda_\e$. 
We prove the stronger inequalities with $\Lambda_0$ replaced by $\Lambda_\e$. 
In view of  (\ref{bound}) and (\ref{triple bound}) we may  replace by $\Lambda_\e$ by $\Lambda_{\e,\R}$. 
Let $\mathcal{E}_\e$ be the 
eight element set of extreme points of $\Lambda_{\e,\R}$. 
For  fixed $a=a_t$  and $\theta=(\theta_1,\theta_2,\theta_3) $ we let $F_\lambda(a,\theta)=\mathcal {K}(\theta) h_1(t_1,\theta_1) h_2(t_2, \theta_2) h_3(t_3, \theta_3)$  and note that 
the assignment $\Lambda_{\e,\R}\to \R_+$, $\lambda\mapsto F_\lambda (a,\theta)$ is convex. Therefore we get  for all 
$\lambda\in \Lambda_\e$ that 

$$ f_\lambda(a) \leq \sum_{\mu \in \mathcal{E}_\e} f_\mu(a)\, .$$
In view of (\ref{triple bound}) the inequalities (\ref{p-bound}) and (\ref{infty-bound}) then follow. 
\par On the other hand for $g=\1=(\1,\1,\1)$, the value
$f_\pi(\1)$ is obtained by applying $I$ to the constant function
$\1=\1\otimes\1\otimes\1$.
This value has been computed explicitly by
Bernstein and Reznikov in \cite{BRe} as 
$$ {\Gamma((\alpha+1)/4)\Gamma((\beta+1)/4)\Gamma((\gamma+1)/4)\Gamma((\delta+1)/4)\over 
\Gamma( (1-\lambda_1)/2)\Gamma( (1-\lambda_2)/2)\Gamma( (1-\lambda_3)/2)}$$
where $\alpha, \beta, \gamma$ are as before  and $\delta= -\lambda_1 -\lambda_2 -\lambda_3$.
Stirling approximation,  
$$
|\Gamma(\sigma+it)|= \text{const.}e^{-\frac\pi2 |t|} |t|^{\sigma-\frac12} 
\left(1+O(|t|^{-1})\right)
$$
as $|t|\to\infty$ and $\sigma$ is bounded, yields a lower bound
for $f_\pi(\1)$:
\begin{equation} \label{inf-2bound} \inf_{\pi \in \Lambda_0} |f_\pi(\1)|>0\, .\end{equation}

As $\|f_\pi\|_\infty\geq |f_\pi(\1)|$
the assertion (\ref{hypo1}) of  Hypothesis A is readily obtained from
(\ref{p-bound}) and (\ref{inf-2bound}). Likewise
(\ref{hypo2})
with $\Omega=\{\1\}$ follows from (\ref{infty-bound}) and (\ref{inf-2bound}).

In general for $G'=\SO_e(1,n)$ one needs to compute the Bernstein-Reznikov
integral. This was accomplished in \cite{Deit}. 

\begin{theorem} \label{thm=errorprasad} Let $Z=G'\times G'\times G'/ \diag (G')$ for
$G'=\SO_e(1,n)$ and assume that $H/\Gamma_H$ is compact.
Then the first error term
$\err(R,\Gamma)$ for the lattice counting
problem on $Z=G/H$ can be estimated as follows: for all $p>p_H(\Gamma)$
there exists a $C=C(p)>0$ such that
$$\err(R,\Gamma)\leq C |B_R|^{-{1\over (6n+3)p}}\, $$
for all $R\geq 1$.
\end{theorem}

\subsection{Cubic lattices}

Here we let $G_0=\SO_e(1,2)$ with the quadratic $Q$ form defining $G_0$
having integer coefficients and anisotropic over $\Q$, for example 
$$Q(x_0, x_1, x_2)= 2x_0^2 - 3x_1^2 - x_2^2\, .$$
Then, according to Borel,   $\Gamma_0=G_0(\Z)$ is a uniform lattice in 
$G_0$. 

\par Next let $k$ be a cubic Galois extension of $\Q$. 
Note that $k$ is totally real. An example of $k$ is the splitting field
of the polynomial $f(x)=  x^3 + x^2 - 2x - 1$.  Let $\sigma$ be a generator of 
the Galois group of $k|\Q$.  Let ${\mathcal O}_k$ be the ring of algebraic 
integers of $k$.  We define $\Gamma<G=G_0^3$ to be the image 
of $G_0({\mathcal O}_k)$ under the embedding 

$$ G_0({\mathcal O}_k)\ni \gamma \mapsto (\gamma, \gamma^\sigma, 
\gamma^{\sigma^2})\in G\, .$$
Then $\Gamma<G$ is a uniform irreducible lattice with trace $H\cap \Gamma\simeq \Gamma_0$
a uniform lattice in $H\simeq G_0$.

\section{Outlook}
We discuss some topics of harmonic analysis on reductive homogeneous spaces which 
are currently open and would have immediate applications to lattice counting.

\subsection{A conjecture which implies Hypothesis A}\label{sec=conj}

Hypothesis A falls in the context of a more general 
conjecture about the growth behavior of families of Harish-Chandra modules. 
\par We let $Z=G/H$ be a real spherical space. 
Denote by $A_Z^-\subset A_Z$ the compression cone of $Z$ (see 
Section \ref{wfrss}) and recall that wavefront means 
that $A^- A_H/A_H = A_Z^-$ which, however, we do not assume for the moment.

\par We use $V$ to denote Harish-Chandra modules for the pair $(\gf, K)$ and 
$V^\infty$ for their unique moderate growth smooth Fr\'echet globalizations.  
These $V^\infty$ are global objects in the sense that they are $G$-modules whereas
$V$ is defined in algebraic terms. We write $V^{-\infty}$ for the strong dual
of $V^\infty$.  We say that $V$ is $H$-distinguished provided that the 
space of $H$-invariants $(V^{-\infty})^H$ is non-trivial.

\par It is no big loss of generality to assume that $A_Z^-$ is a sharp cone, as 
the edge of this cone is in the normalizer of $H$ and in particular acts on the finite 
dimensional space of $H$-invariants.

\par As $A_Z^-$ is pointed it is a fundamental domain for the little Weyl group 
and as such a simplicial cone (see \cite{KK}, Section 9). 
If $\af_Z^-=\log A_Z^-$, then we write $\omega_1, \ldots,\omega_r$ for a set of generators (spherical co-roots) of $\af_Z^-$.

\par Set $\oline Q:=\theta(Q)$ where $\theta$ is the Cartan involution determined by the choice of $K$. 
Note that $V/ \oline \qf V$ is a finite dimensional $\oline Q$ module, in 
particular a finite dimensional  $A_Z$-module.  Let $\Lambda_1, \ldots, \Lambda_N\in \af_Z^*$ be the $\af_{Z,\C}$-weight spectrum. 
Then we define the $H$-spherical exponent $\Lambda_V\in \af_Z^*$ of $V$ by
$$ \Lambda_V(\omega_i) :=\max_{1\leq j \leq N} \re \Lambda_j(\omega_i)\, .$$

\par Further attached to $V$ is a ``logarithmic'' exponent $d\in \N$. Having this data 
we recall the main bound from \cite{KSS2} 
$$ |m_{v, \eta} (a\cdot z_0)| \lessapprox a^{\Lambda_V} ( 1 + \|\log a\|)^{d_V} \qquad  
(a \in A_Z^-)\, . $$

\begin{con} Fix a $K$-type $\tau$, a constant $C>0$, and a compact subset $\Omega\subset G$.
Then there exists a compact 
set $\Omega_A \subset A_Z^-$  such that for all Harish-Chandra modules $V$ with 
$\|\Lambda_V\|\leq C$, all $v \in V[\tau]$ and all $\eta \in (V^{-\infty})^{H}$ one has 
\begin{align*} &\max_{a\in A_Z^-\atop g\in \Omega} |m_{v, \eta}(ga\cdot z_0)| a^{-\Lambda_V} 
( 1 + \|\log a\|)^{-d_{V}}  = \\
& \max_{a\in \Omega_A\atop g\in \Omega} |m_{v, \eta}(ga\cdot z_0)| a^{-\Lambda_V} 
( 1 + \|\log a\|)^{-d_{V}}\, .\end{align*}
\end{con}

It is easily seen that this conjecture implies Hypothesis A
if all the generalized matrix coefficients $m_{v,\eta}$
are bounded, as for example it is the case when $Z$ is wavefront (see Proposition \ref{propI}(\ref{one})).

\begin{rmk} It might well be that a slightly stronger conjecture is true. For that we recall that a Harish-Chandra module 
$V$ has a unique minimal globalization, the analytic model $V^\omega$. The space  $V^\omega$ is an increasing 
union of subspaces $V_\e$ for $\e \to 0$. The parameter $\e$ parametrizes left $G$-invariant neigborhoods $\Xi_\e\subset G_\C$
of $\bf 1$ which decrease with $\e\to 0$. Further $V_\e$ consists of those vectors $v\in V^\omega$ for which the orbit 
map $G\to V^\omega, \ \ g\mapsto g\cdot v$ extends to a holomorphic map on $\Xi_\e$.
For fixed $\e, C>0$ the strengthened conjecture would be that there 
exists a compact subset $\Omega_A$ such that  for all Harish-Chandra modules $V$ with 
$\|\Lambda_V\|\leq C$ and all $v \in V_\e$ one has 
\begin{align*} &\max_{a\in A_Z^-} |m_{v, \eta}(a\cdot z_0)| a^{-\Lambda_V} ( 1 + \|\log a\|)^{-d_{V}}  = \\
& \max_{a\in \Omega_A} |m_{v, \eta}(a\cdot z_0)| a^{-\Lambda_V} ( 1 + \|\log a\|)^{-d_{V}}\, .\end{align*}
Note that the compact set $\Omega$ is no longer needed, as $\Omega \cdot V_\e \subset V_{\e'}$.
\end{rmk}

\subsection{Spectral geometry of $Z_\eta$}
In the general context of a reductive real spherical space 
it may be possible to establish both main term counting and the error term bound,
with the arguments presented here for wavefront spaces,  
provided the following two key questions allow affirmative answers.

\par In what follows $Z=G/H$ is a real reductive spherical space and
$V$ denotes an irreducible  Harish-Chandra module and $\eta \in (V^{-\infty})^H$.
\bigskip 
\par \noindent {\bf Question A:}  {\it Is $H_\eta$ reductive?}
\par 
\medskip
\noindent {\bf Question B:}  {\it If for $v\in V$ the generalized matrix coefficient $m_{v, \eta}$ is bounded, then there exists a
$1\leq p <\infty$ such that $m_{v,\eta} \in  L^p (Z_\eta)$. }
\bigskip 

In this context we note that issues related to the well-factorization of the
intrinsic balls in affine spherical spaces can possibly be resolved with similar methods to
those applied here, using volume estimates as described in Theorem 7.17 of \cite{GNbook}.

\smallskip{\bf Acknowledgement.} We are grateful to an anonymous referee for valuable comments
that improved the presentation in the present version of the paper.

\end{document}